\documentclass[11pt,letterpaper]{amsart}

\usepackage[T1]{fontenc}
\usepackage{lmodern}
\usepackage{etex}
\usepackage[shortlabels]{enumitem}
\usepackage{amsmath}
\usepackage{amsxtra}
\usepackage{amscd}
\usepackage{amsthm}
\usepackage{amsfonts}
\usepackage{amssymb}
\usepackage{booktabs, tabularx}
\usepackage{eucal}
\usepackage[all]{xy}
\usepackage{graphicx}
\usepackage{tikz-cd}
\usepackage{mathrsfs}
\usepackage{subfiles}
\usepackage{mathpazo}
\usepackage[colorinlistoftodos, textsize=tiny]{todonotes}
\usepackage{morefloats}
\usepackage{pdfpages}
\usepackage{thm-restate}
\usepackage[utf8]{inputenc}
\usepackage{epigraph}
\usepackage{csquotes}
\usepackage[margin=1in]{geometry}
\usepackage{adjustbox}
\usepackage{scalerel}
\usepackage{stackengine}
\stackMath
\newcommand\reallywidehat[1]{%
\savestack{\tmpbox}{\stretchto{%
  \scaleto{%
    \scalerel*[\widthof{\ensuremath{#1}}]{\kern-.6pt\bigwedge\kern-.6pt}%
    {\rule[-\textheight/2]{1ex}{\textheight}}%WIDTH-LIMITED BIG WEDGE
  }{\textheight}% 
}{0.5ex}}%
\stackon[1pt]{#1}{\tmpbox}%
}
\graphicspath{ {images/} }

\RequirePackage{color}
\definecolor{myred}{rgb}{0.75,0,0}
\definecolor{mygreen}{rgb}{0,0.5,0}
\definecolor{myblue}{rgb}{0,0,0.65}

\usepackage{color}

\usepackage[
	hypertexnames=false,
    pdftex, 
 	pdfpagemode=UseNone,
 	breaklinks=true,
 	extension=pdf,
 	colorlinks=true,
 	linkcolor=blue,
 	citecolor=red,
 	urlcolor=blue,
 ]{hyperref} 
 \usepackage{cleveref}

\usepackage{tikz}
\usetikzlibrary{matrix,arrows,decorations.pathmorphing}

\theoremstyle{plain}
\newtheorem{theorem}{Theorem}[section]
\newtheorem*{maintheorem}{Main theorem}

\newtheorem{proposition}[theorem]{Proposition}
\newtheorem{lemma}[theorem]{Lemma}
\newtheorem{corollary}[theorem]{Corollary}

\theoremstyle{definition}
\newtheorem{definition}[theorem]{Definition}
\newtheorem{remark}[theorem]{Remark}

\newtheorem{example}[theorem]{Example}

\newtheorem*{questionintro}{Question}

\theoremstyle{remark}

\numberwithin{equation}{section}
\newcommand\nc{\newcommand}
\nc\on{\operatorname}
\nc\renc{\renewcommand}

\newcommand\dR{\on{dR}}

\newcommand\bc{\mathbb C}

\DeclareMathOperator\Pic{Pic}

\newcommand*{\shom}{\mathscr{H}\kern -.5pt om}
\newcommand*{\stor}{\mathscr{T}\kern -.5pt or}
\newcommand*{\sext}{\mathscr{E}\kern -.5pt xt}

\makeatletter
\providecommand\@dotsep{5}
\renewcommand{\listoftodos}[1][\@todonotes@todolistname]{%
\@starttoc{tdo}{#1}}
\makeatother

\makeatletter
\newcommand{\customlabel}[2]{\protected@write \@auxout {}{\string \newlabel {#1}{{#2}{\thepage}{#2}{#1}{}} }\hypertarget{#1}{#2}}

\DeclareMathOperator\id{id}

\DeclareMathOperator\rk{rk}

\DeclareMathOperator\im{im}
\DeclareMathOperator\End{End}

\DeclareMathOperator\pr{pr}

\DeclareMathOperator\inv{inv}

\DeclareMathOperator\Gr{Gr}

\DeclareMathOperator{\Hom}{Hom}
\DeclareMathOperator{\Sym}{Sym}
\DeclareMathOperator{\cHom}{\mathcal{H}\textit{om}}

\newcommand{\bbC}{\mathbb{C}}

\newcommand{\bbL}{\mathbb{L}}

\newcommand{\bbN}{\mathbb{N}}

\newcommand{\bbP}{\mathbb{P}}
\newcommand{\bbQ}{\mathbb{Q}}
\newcommand{\bbR}{\mathbb{R}}

\newcommand{\bbV}{\mathbb{V}} 
\newcommand{\bbW}{\mathbb{W}}

\newcommand{\bbZ}{\mathbb{Z}}

\renewcommand{\mathcal}{\mathscr}

\newcommand{\cE}{\mathcal{E}}
\newcommand{\cF}{\mathcal{F}}

\newcommand{\cH}{\mathcal{H}}

\newcommand{\cL}{\mathcal{L}}
\newcommand{\cM}{\mathcal{M}}
\newcommand{\cN}{\mathcal{N}}
\newcommand{\cO}{\mathcal{O}}
\newcommand{\cP}{\mathcal{P}}

\newcommand{\cV}{\mathcal{V}}
\newcommand{\cW}{\mathcal{W}}
\newcommand{\cX}{\mathcal{X}}
\newcommand{\cY}{\mathcal{Y}}

\newcommand{\rH}{\textup{H}}

\newcommand{\rR}{\textup{R}}

\DeclareMathOperator{\Alb}{Alb}

\DeclareMathOperator{\Gal}{Gal}
\DeclareMathOperator{\GL}{GL}
\DeclareMathOperator{\SL}{SL}

\DeclareMathOperator{\Sp}{Sp}

\DeclareMathOperator{\SO}{SO}

\DeclareMathOperator{\gr}{gr}
\newcommand{\iso}{\simeq}
\newcommand{\into}{\hookrightarrow}
\DeclareMathOperator{\Lie}{Lie}

\newcommand{\rd}{\textup{d}}
\newcommand{\too}{\longrightarrow}

\newcommand{\Perv}{\textup{Perv}}

\DeclareMathOperator{\Vect}{Vect}

\DeclareFontFamily{U}{wncy}{}
\DeclareFontShape{U}{wncy}{m}{n}{<->wncyr10}{}
\DeclareSymbolFont{mcy}{U}{wncy}{m}{n}
\DeclareMathSymbol{\Sha}{\mathord}{mcy}{"58}

\def\listtodoname{List of Todos}
\def\listoftodos{\@starttoc{tdo}\listtodoname}

\title{$E_6$-local systems from cubic threefolds}

\author{Thomas Kr\"amer, Daniel Litt, Marco Maculan}

\usepackage{microtype}

\begin{document}

\begin{abstract} 
We produce infinitely many local systems on (level covers of) the moduli space of smooth cubic threefolds, with algebraic monodromy group equal to the exceptional group $E_6$. These local systems arise in the middle cohomology of abelian \'etale covers of the Fano scheme parametrizing lines in the universal cubic threefold.
\end{abstract}

\maketitle

\setcounter{tocdepth}{1}
\tableofcontents

\thispagestyle{empty}

\section{Introduction}\label{sec:intro}

The existence question for motives over number fields with exceptional Galois group was first raised for the groups $G_2$ and $E_8$ by Serre \cite[§8.8]{serre1994proprietes}, who qualified it as \emph{hasardeuse}. 
His interest probably came from the fact that constructions based on curves and abelian varieties typically yield only classical groups. It is now known that all simple exceptional groups do occur as Galois groups of motives. In this paper, we consider the following version of the problem:

\begin{questionintro}
Is every simple exceptional group realized as the algebraic monodromy group of a local system of geometric origin on a smooth complex variety?
\end{questionintro}

This is not merely an analogue of Serre's question: by a specialization argument, a positive answer also resolves the original question on motives if one interprets the motivic Galois group either as an $\ell$-adic monodromy group or as a Mumford--Tate group.
As we will recall in \cref{sec:previous-work}, known constructions yield the sought-after local systems for all simple exceptional groups other than $E_6$. The main goal of this paper is to fill in this missing case: more precisely, we construct infinitely many pairwise distinct local systems of geometric origin whose algebraic monodromy group is the exceptional group \(E_6\). By specializing these, we obtain many motives with Galois group \(E_6\). The local systems that we construct arise in the middle cohomology of abelian \'etale covers of the family of Fano surfaces associated with a family of smooth cubic threefolds, as we now explain.

\subsection{Main result} \label{sec:main-results}
Let $S$ be a smooth complex variety and $\cY \to S$ a family of smooth cubic threefolds, i.e.~a subvariety
\[ \cY \subset \bbP(\cE) \]
cut out by a nonzero global section of $\Sym^3 \cE^\vee$ for a vector bundle $\cE$ of rank five on $S$ such that the projection $\cY \to S$ is a smooth morphism. Let~$\cY_s$ be the fiber of this morphism at a point $s\in S$. The family is called \emph{versal} if for general $s \in S$ the Kodaira--Spencer map 
\[ T_s S \to \rH^1(\cY_s, T_{\cY_s}) \]
is surjective, or equivalently if the classifying map $S \to \cM$ to the moduli stack~$\cM$ of cubic threefolds is dominant. Passing from smooth cubic threefolds to their associated Fano surfaces, we get a smooth projective morphism 
\[ \pi \colon \cF \to S \]
whose fiber $\cF_s$ at a point $s\in S$ is the surface parametrizing lines on the cubic threefold $\cY_s$. 

We are interested in the middle cohomology of this family of surfaces twisted by local systems of rank one. Let $\bbL$ be a unitary complex local system of rank one on $\cF$. By Ehresmann's theorem, the higher direct image
\[ \bbV := \rR^2 \pi_\ast \bbL \]
is a local system on $S$. For $s \in S$, the topological fundamental group $\pi_1(S, s)$ acts on $\bbV_s = \rH^2(F, \bbL)$ via the monodromy representation
$\rho_s \colon \pi_1(S, s) \to \GL(\bbV_s)$.
By the~\emph{algebraic monodromy group} of $\bbV$ we mean the Zariski closure
\[ M_s \; := \; \overline{\im \rho_s} \; \subset \; \GL(\bbV_s) \]
of the image of $\rho_s$. Since $\bbL$ is unitary, the connected component of the identity~$M_s^\circ$ is a reductive Lie group. We will consider the case where $\bbL$ is of \emph{ finite order} in the sense that for some $n>0$ the local system~$\bbL^{\otimes n}$ is trivial. The smallest such $n$ is called the \emph{order} of $\bbL$. To compare the monodromy representations arising as above for different choices of~$\bbL$, recall that the \emph{trace field} $\bbQ(\rho_s)\subset \bbC$ is generated over $\bbQ$ by the traces of $\rho_s(\gamma)$ for all~$\gamma \in \pi_1(S, s)$. The trace field may change if one replaces $S$ by a finite \'etale cover. To take this into account, define the~\emph{invariant trace field} as the intersection
\[ \inv(\rho_s) \; := \; \bigcap_{\Gamma} \bbQ(\rho_s\rvert_{\Gamma}) \]
over all finite index subgroups $\Gamma \subset \pi_1(S, s)$. We can then state our main result as follows.

\begin{maintheorem}\label{thm:main-thm} 
There exists an integer $n_0 > 2$ such that for any versal family $\cY \to S$ of smooth cubic threefolds, the following holds:
\begin{enumerate}
\item For any local system $\bbL$ of finite order whose restriction to a general fiber $F$ of $\pi\colon \cF\to S$ has order~$n \ge n_0$, we have
\[ (M^\circ_s, \bbV_s) \iso (E_6, V) \quad \text{and} \quad \inv(\rho_s) = \bbQ(\zeta_n) \]
where $V$ is one of the two $27$-dimensional irreducible representations of $E_6$ and $\zeta_n = e^{2 \pi i /n}$.\smallskip 
\item For $i=1,2$ let $\bbL_i$ be a local system of rank one with the above properties and put $\bbV_i := \rR^2 \pi_\ast \bbL_i$. If the two local systems $\bbV_1$ and $\bbV_2$ are isomorphic, then so are $\bbL_1 \rvert_F$ and $\bbL_2 \rvert_F$. 
\end{enumerate}
\end{maintheorem}
 
Note that the Fano surface $F=\cF_s$ of lines on a smooth cubic threefold has 
$\dim \rH_1(F, \bbQ) = 10$, so for any $n \ge 1$ there are many rank one local systems of order $n$ on it. After replacing~$S$ by a finite level cover, they all extend to local systems $\bbL$ on the total space of our family.

The local systems $\bbV$ that we produce are distinct in the following strong sense. Let us say that a local system $\bbW_1$ on a smooth variety $T_1$ is \emph{equivalent} to a local system $\bbW_2$ on a smooth variety $T_2$ if there is a smooth variety $T$ with dominant morphisms $p_i \colon T\to T_i$ such that $p_1^\ast \bbW_1 \iso p_2^\ast \bbW_2$. If this is the case, then the local systems $\bbW_1$ and $\bbW_2$ have the same invariant trace field. Hence the local systems $\bbV$ arising from two distinct values  $n_1, n_2$ of $n$ in part (1) of the theorem are non-equivalent as long as $\mathbb{Q}(\zeta_{n_1})\neq \mathbb{Q}(\zeta_{n_2})$, which holds for any pair of distinct integers $n_1,n_2$ unless one of them is an odd number and the other is twice that number.

Note that in the formulation of the main theorem we must have $n_0>2$. Indeed, if the restriction of $\bbL$ to a general fiber of $\pi$ has order $n = 2$, then the local system $\bbV$ of rank~$27$ on~$S$ is self-dual up to twist by a rank one local system. In this case $\bbV$ cannot have connected monodromy group $E_6$ since the $27$-dimensional irreducible representations of $E_6$ are not self-dual. We haven't computed the connected monodromy group for $n=2$ but expect that it is the symplectic group $\Sp_8$ and acts on~$\bbV_s$ via the nontrivial summand in the exterior square of the natural $8$-dimensional representation.

\subsection{Previous work} \label{sec:previous-work}

As we already mentioned at the beginning, all simple exceptional groups are known to occur as Galois groups of motives (operationalized, for example, as $\ell$-adic monodromy groups or Mumford-Tate groups):

\begin{itemize}
\item[$G_2\colon$]
Dettweiler and Reiter \cite{dettweiler2010rigid} constructed a rigid rank $7$ local system on $\mathbb{P}^1 \smallsetminus \{0,1,\infty\}$ with monodromy group $G_2$ via Katz’s middle convolution operation. It is of geometric origin, and its specializations yield motives of type $G_2$.
Earlier, Gross and Savin \cite{gross1998motives} proposed a conjectural construction of such motives in the cohomology of a Shimura variety. Yun \cite{yun2014motives} later constructed exceptional motivic local systems whose $G_2$ case is conjecturally equivalent to that of Dettweiler--Reiter.
\smallskip

\item[$E_6\colon$]
Patrikis \cite{patrikis2016deformations, patrikis2017deformations} constructed 
$p$-adic Galois representations with Zariski dense image in~$E_6$
which are geometric in the sense of Fontaine--Mazur, i.e.~unramified almost everywhere and de Rham at primes above $p$. Such Galois representations are conjectured to arise in the $p$-adic \'etale cohomology of an algebraic variety. Boxer, Calegari, Emerton, Levin, Madapusi Pera, and Patrikis \cite{boxer2019compatible} subsequently realized $E_6$ as the Galois group of a motives occurring in the cohomology of a Shimura variety, though not a family of such.
\smallskip

\item[$E_7, E_8\colon$]
Yun \cite{yun2014motives}, using ideas from the Langlands program, constructed motivic local systems on $\mathbb{P}^1 \smallsetminus \{0,1,\infty \}$ with monodromy groups $E_7$ and $E_8$, whose specializations yield motives of these types.
\smallskip

\item[$F_4\colon$]
Patrikis \cite{patrikis2016deformations, patrikis2017deformations} also found geometric $p$-adic Galois representations with Zariski dense image in $F_4$. Independently, Guralnick, L\"{u}beck, and Yu \cite{guralnick2016rational} constructed rigid $F_4$-local systems on $\mathbb{P}^1 \smallsetminus \{0,1,\infty \}$. F\ae rgeman \cite{faergeman2024motivic} recently showed that all rigid $G$-local systems are motivic.
\end{itemize}

For the groups $G=G_2, E_7, E_8, F_4$, the first, third, and fourth constructions above also provide an answer to the geometric analogue of Serre's question that we study here: that is, they provide local systems of geometric origin with algebraic monodromy group $G$. To the best of our knowledge, this geometric analogue has remained open for $E_6$ prior to this work. Note that there are natural  candidates for motivic local systems with monodromy $E_6$ or $E_7$, namely the tautological local systems on Shimura varieties of type $E_6$ and $E_7$, respectively. Although these are not known to be of geometric origin, Diao, Lan, Liu, and Zhu \cite{diao2023logarithmic} show that they are geometric in the sense of Fontaine--Mazur, and are therefore conjectured to arise from geometry. In particular, by suitable specialization, one can construct Galois representations of the desired type that are conjecturally motivic. In contrast, the local systems constructed in this paper are not of Shimura type.

One of the goals of this paper is to explain how the remarkable geometry of cubic threefolds reflects the exceptional nature of the group $E_6$. It would be interesting to find similarly explicit geometric constructions for the other exceptional groups.

\subsection{Ideas of the proof}

Before going into details, let us briefly explain the main ideas in the proof that the local system~$\bbV$ of the main theorem has connected monodromy group $M_s^\circ = E_6$. We proceed in two steps: first we observe that there is an inclusion $M_s^\circ \subset E_6$, and then we show that this inclusion cannot be strict.

The inclusion $M_s^\circ \subset E_6$ follows directly from results in the literature. Indeed, Lawrence and Sawin \cite[lemma~2.9]{lawrence2025shafarevich} prove that the monodromy normalizes the Tannaka group $G$ associated with the convolution of perverse sheaves on the abelian variety $\Alb(\cF_{\bar{\eta}})$, where $\cF_{\bar{\eta}}$ is a geometric generic fiber of $\pi \colon \cF \to S$, see also \cite[th.~4.5]{JKLM}. On the other hand, by~\cite[th.~2]{kramer2016cubic} the derived group of the connected component of the identity in $G$ is $E_6$ acting via one of its two irreducible $27$-dimensional representations (as explained in~\cite{kramercharacteristicI}, this is related to the fact that the monodromy of the family of lines in the universal cubic \emph{surface} is the Weyl group $W(E_6)$). Schur's lemma then gives the desired inclusion $M_s^\circ \subset E_6$.

The main task of this paper is to prove that this inclusion cannot be strict. Note that the comparison results between algebraic monodromy groups and Tannaka groups in \cite[th.~4.7]{lawrence2025shafarevich} and~\cite[th.~4.10]{JKLM} do not apply here because the abelian scheme $\on{Alb}(\mathscr{F}/S)$ is not constant. In contrast, our proof is Hodge-theoretic in nature and builds on a higher-dimensional generalization of the theory developed in \cite{landesman2025big} to prove big monodromy results for curves.

Since $M_s^\circ$ is the connected monodromy group of a local system of geometric origin, it is necessarily semisimple and it therefore suffices to show that $M_s^\circ$ cannot be contained in any of the maximal semisimple subgroups of $E_6$. Our proof of this relies on an analysis of the derivative of the period map associated to $\bbV$ at a general point of $S$, as we now explain.

Recall that the surface $F:=\cF_s$ parametrizes lines on the cubic threefold $Y:=\cY_s$. A line on~$Y$ is called of \emph{second type} if its normal bundle splits as $\cO(1) \oplus \cO(-1)$. The locus of lines of second type is an effective bicanonical divisor $B\subset F$. 
We first prove a reconstruction result for the restriction to~$B$ of the line bundle $\cL = \bbL \otimes \cO_\cF$ attached to our local system: if $\cL \otimes \omega_F$ and $\cL^\vee \otimes \omega_F$ are globally generated (which is true for general $s$ and local systems of high enough order), then we may functorially reconstruct~$\cL|_B$ from the second derivative of the period map associated to the complex variation of Hodge structures associated to $\bbV$. Since $B \subset F$ is an ample divisor, this is enough to functorially reconstruct the unitary local system $\bbL|_F$ by the Narasimhan--Seshadri correspondence. We deduce from this functorial reconstruction that 
\begin{enumerate} 
\item $\bbV$ is not self-dual, and \smallskip 
\item $\bbV \simeq \bbW \oplus \bbW'$ where $\bbW$ is irreducible of Hodge length $2$, rank $ \ge 13$, and $\bbW'$ is unitary. 
\end{enumerate} 
It turns out that these two properties suffice to show that the connected monodromy group $M_s^\circ$ cannot be contained in any of the maximal semisimple subgroups of $E_6$, using a case by case analysis based on the branching of the $27$-dimensional irreducible representation.

\subsection{Acknowledgments} We thank Claire Voisin and Frank Gounelas for useful discussions. In the process of improving a first draft of this paper, we produced the proofs of \cref{thm:global-generation-line-bundle} and \cref{prop:kernel-cup-product} with help from Deepmind's Aletheia scaffold and Gemini Deep Think, respectively.

%%%%%%%%%%%%%%%%%%%%%%%%%%%%%%%%%%%
%
% RECONSTRUCTION SECTION
%
%%%%%%%%%%%%%%%%%%%%%%%%%%%%%%%%%%%

\section{The reconstruction technique}\label{sec:ivhs}

In this section we explain how to recover properties of an integrable connection on a smooth proper family from the local system underlying its relative de Rham cohomology, using a higher-dimensional version of the method developed in \cite{landesman2025big} for families of curves.

\subsection{The derivative of the period map} 

Let $S$ be a smooth irreducible complex variety and $\pi \colon \cX \to S$ a smooth projective morphism of relative dimension $d$. Let $\cL$ be a vector bundle on $\cX$ endowed with a flat unitary connection $\nabla \colon \cL \to \cL \otimes \Omega^1_{\cX}$. For any point $s\in S$, we can consider the de Rham cohomology 
\[
 \rH^\bullet_{\dR}(X/\bbC,(\cL, \nabla))
 \quad \text{of the fiber} \quad X \;:=\; \cX_s.
\]
It comes with a natural Hodge filtration. For all $p , q \ge 0$ the coherent sheaves $\rR^q \pi_\ast (\cL \otimes \Omega^p_{\cX / S})$ are locally free at~$s$ and, since $\nabla$ is unitary, we have an isomorphism
\[ \gr^p \rH^{p + q}_{\dR} (X / \bbC, (\cL, \nabla)) \iso \rH^q (X,  \cL \otimes \Omega^p_{X}) . \]
We are interested in the cohomology in degree $d=\dim(X)$. Varying the point $s\in S$, we consider the relative de Rham cohomology 
\[ \cV := \cH_{\textup{dR}}^d(\cX/S, (\cL, \nabla))=\bbR^d \pi_*(\Omega^\bullet_{\cX/S} \otimes \cL), \]
where $\Omega_{\cX/S}^\bullet \otimes \cL$ is the de Rham complex of the flat bundle $\cL$. Let $\cF^\bullet$ be the Hodge filtration on~$\cV$ and $\nabla\colon \cV\to \cV\otimes \Omega^1_S$ the Gauss--Manin connection. The Hodge filtration satisfies the Griffiths transversality condition:
\[ \nabla(\cF^p) \; \subset \; \cF^{p - 1} \otimes \Omega^1_{S}
\quad \text{for all $p\ge 0$}.
\]
On the graded pieces of the Hodge filtration, the Gauss--Manin connection gives rise to $\cO_S$-linear maps
\[ \gr^p \nabla \colon \gr^p \cV \to \gr^{p - 1} \cV \otimes \Omega^1_S \]
which make $\bigoplus_{p\ge 0} \gr^p \cV$ a graded Higgs bundle in the following sense:

\begin{definition} A Higgs bundle $\cH$  on $S$ with Higgs field $\phi$ is \emph{graded} if the underlying vector bundle comes with a grading $\cH = \bigoplus_{p \in \bbZ} \cH^p$ such that
\[ \phi(\cH^p) \subset \cH^{p-1} \otimes \Omega^1_S \qquad \text{for all $p \in \bbZ$}.\]
\end{definition}
Fix now $s \in S$. The goal of this section is to recover information about the vector bundle $\cL_{\vert X}$ on the fiber $X := \cX_s$ from the fiber of the Higgs bundle $\gr \cV$ at $s$. By taking the fiber at $s$ of a graded Higgs bundle on $S$, one obtains the following linear datum:

\begin{definition}An \emph{infinitesimal Higgs bundle} at $s$ is a finite-dimensional vector space $H$ together with a linear map
\[
\phi \colon H \to H \otimes \Omega^1_{S,s},
\]
called the \emph{Higgs field}, such that $\phi \wedge \phi = 0$. We say that the infinitesimal Higgs bundle is~\emph{graded} if the underlying vector space is endowed with a grading
\(
H = \bigoplus_{p \in \bbZ} H^p
\)
such that
\[
\phi(H^p) \subset H^{p-1} \otimes \Omega^1_{S,s}
\quad \text{for all } p \in \bbZ.
\]
We will omit the Higgs field $\phi$ from the notation if there is no risk of confusion.
\end{definition}

Any Higgs bundle on $S$ gives rise to an infinitesimal Higgs bundle by taking the fiber at $s$. Given an infinitesimal  Higgs bundle $H$ with Higgs field $\phi$, it will be often convenient to consider the adjoint \[\theta \colon T = T_{S,s} \to \End(H)\] to the Higgs field $\phi$, called \emph{adjoint Higgs field}.  For the adjoint Higgs field the condition $\phi \wedge \phi = 0$ reads
\[ \theta(\delta_1) \circ \theta(\delta_2) = \theta(\delta_2) \circ \theta(\delta_1)\]
for all $\delta_1, \delta_2 \in T$. Note that when $H$ is graded, the adjoint Higgs field is a map
\[ \theta \colon T  \to  \bigoplus_{p \in \bbZ} \Hom(H^p, H^{p-1})\]
 In particular, for every $n\in \bbN$, the $n$-th iterate of $\theta$ induces a linear map
\[  
\theta^{\circ n}\colon \quad \Sym^n T \to \bigoplus_{p\in \bbZ} \Hom(H^p, H^{p-n}),\qquad \delta_1 \cdots \delta_n \mapsto \theta(\delta_1) \circ \cdots \circ \theta(\delta_n).
\]
Morphisms and direct sums of infinitesimal Higgs bundles, and their graded version, are defined in the evident manner. The dual of an infinitesimal Higgs bundle $H$ is also defined in the obvious way, while for the graded version we will describe our weight conventions later.

\begin{example} \label{ex:cup-product} When $H$ is the fiber at $s$ of the Higgs bundle $\gr \cV$, we write
\[
H := H_{\cL},
\]
leaving the connection $\nabla$ understood when no confusion can arise. With this notation, we have
\[
H_{\cL}^p \;\simeq\; \rH^{d-p}(X, \cL \otimes \Omega_X^{p}).
\]
Moreover, the adjoint Higgs field
\[
\theta \colon T = T_{S,s} \longrightarrow \bigoplus_{p \ge 0} \Hom(H_\cL^p, H_\cL^{p-1})
\]
can be interpreted as the differential at $s$ of the period map of
\(
\cH_{\dR}^d(\cX/S, (\cL,\nabla)),
\)
although we will not use this interpretation in what follows. Instead, we use that, via the Kodaira–Spencer map
\[
\kappa \colon T \longrightarrow \rH^1(X, T_X),
\]
one can express $\theta$ in terms of the cup product:
\[
\langle \theta(t), \alpha \rangle
\;=\;
\alpha \cup \kappa(t)
\quad
\text{for all } t \in T
\text{ and } 
\alpha \in H^p \simeq \rH^{d-p}(X, \cL \otimes \Omega_X^{p}),
\]
see~\cite[th.~3.5]{KatzIHES}. We may therefore factor the $d$-th iterate of the Higgs field as shown in the commutative diagram
\begin{equation}\label{factorization-derivative-of-period-map}
\begin{tikzcd}
\Sym^d T \ar[d, "\Sym^d \kappa"'] \ar[r, "\theta^{\circ d}"]
  & \Hom(H_\cL^d, H_\cL^0) \\
\Sym^d \rH^1(X, T_X) \ar[r, "\mu"]
  & \rH^d(X, \bigwedge^d T_X) \ar[u, "c"]
\end{tikzcd}
\end{equation}
where $\mu$ is induced by the cup product and
\[
c \colon \rH^d(X, \omega_X^\vee)
\longrightarrow
\Hom(H_\cL^d, H_\cL^0),
\qquad
v \longmapsto \bigl[\alpha \longmapsto \alpha \cup v\bigr].
\]
Note that the anti-symmetry of the cup product on $\rH^1(X, T_X)$ exchanges the symmetric power in the source of $\mu$ with the exterior power in its target. We also observe that the lower-left part of the commutative square~\eqref{factorization-derivative-of-period-map} does not depend on the choice of the flat bundle $\cL$.
\end{example}

The previous example leads us to the following definition:

\begin{definition} \label{first-order-VHS} An infinitesimal graded Higgs bundle $H$ with adjoint Higgs field $\theta$ is \emph{compatible with $\pi \colon \cX \to S$} if
\[ \ker (\mu \circ \Sym^d \kappa) \subset \ker (\theta^{\circ d} \colon \Sym^d T \to \bigoplus_{p \in \bbZ}\Hom(H^{p}, H^{p- d})) \]
The above condition is equivalent to the existence of a  linear map $c \colon \im(\mu) \to \bigoplus_{p} \Hom(H^p, H^{p - d})$ making the following diagram commutative:
\begin{equation} \label{factorization-abstract-higgs-field}
\begin{tikzcd}
\Sym^d T \ar[r, "\theta^{\circ d}"] \ar[d, swap, "\Sym^d \kappa"]& \bigoplus_{p}\Hom(H^{p}, H^{p- d}) \\
\Sym^d \rH^1(X, T_X) \ar[r, "\mu"] & \im(\mu). \ar[u, "c"]
\end{tikzcd}
\end{equation}
\end{definition}
Let $H$ be an infinitesimal graded Higgs bundle which is compatible with $\pi$ and let $H^\vee$ be its dual as an infinitesimal Higgs bundle. We equip it with the grading
\[ (H^\vee)^p = (H^{d - p})^\vee, \]
so that the resulting infinitesimal graded Higgs bundle is compatible with $\pi$. When  $H = H_{\cL}$ is the fiber at $s$ of the Higgs bundle $\gr \cV$ we have
\[ (H_{\cL})^\vee \simeq H_{\cL^\vee}.\]

\subsection{The reconstruction result} \label{sec:reconstruction}

From now on we assume that the family $\pi \colon \cX \to S$ is \emph{versal} in the sense that the Kodaira--Spencer map $\kappa$ is surjective at $s$. Note that in this case the map $c$ in~\eqref{factorization-abstract-higgs-field} is determined uniquely by the $d$-th power of the adjoint Higgs field.

The goal of this section is to recover information about the vector bundle $\cL_{\vert X}$ on the fiber $X=\cX_s$ from the relative de Rham cohomology on a first order neighborhood of the point $s\in S$, or more precisely from the infinitesimal graded Higgs bundle $H_\cL$ in~\cref{ex:cup-product}. The idea is to recover information about $\cL_{\vert X}$ from the cup product map
\[
 c\colon \quad \rH^d(X, \textstyle \bigwedge^d T_X) \;\to\; \Hom(H_\cL^d, H_\cL^0)
\]
and to reconstruct the latter from the $d$-th iterate of the Higgs field via the diagram~\eqref{factorization-derivative-of-period-map}. Note that the diagram only determines $c$ on the image of the map \[\mu:\Sym^d \rH^1(X, T_X)\to  \rH^d(X, \textstyle \bigwedge^d T_X).\] To take care of this, let us rewrite the target of $\mu$ as
\[ \rH^d(X, \textstyle \bigwedge^d T_X) \; \simeq \; \rH^0(X, \omega_X^{\otimes 2})^\vee \]
via Serre duality. Then the kernel of the dual map $\mu^\vee$ becomes a subspace
$\ker(\mu^\vee) \subset \rH^0(X, \omega_X^{\otimes 2})$, and we denote by
\[ B \; \subset \; X\]
the base locus of the associated linear system. The evaluation map
$ \rH^0(X, \omega_X^{\otimes 2}) \otimes \cO_B \to \omega_X^{\otimes 2}$ factors through a map 
\begin{equation} \label{evaluation-on-B} 
\im(\mu^\vee) \otimes \cO_B \to \omega_X^{\otimes 2} \rvert_B 
\end{equation}
via the natural identification $\im(\mu^\vee)%\simeq \im(\mu)^\vee 
\simeq 
\rH^0(X, \omega_X^{\otimes 2}) / \ker(\mu^\vee)$. 

\begin{definition} Let $H$ be an infinitesimal graded Higgs bundle which is compatible with $\pi$. We 
consider the graded coherent sheaf $\cE(H):= \im \psi_H$ where $\psi_H = \bigoplus_p \gr^p \psi_H$ and $\gr^p \psi_H$ is
the composite morphism
\[ \gr^p \psi_{H} \colon H^p \otimes \omega_X^\vee \rvert_B 
\longrightarrow H^{p-d} \otimes \im(\mu^\vee) \otimes \omega_X^\vee \rvert_B 
\longrightarrow H^{p-d} \otimes \omega_X \rvert_B. \]
The first arrow above is given by the map $H^p \to H^{p - d} \otimes \im(\mu^\vee)$ adjoint to $c$ and the second is given by the evaluation map \eqref{evaluation-on-B}. Since all the maps in question depend functorially on $H$, this gives a functor
\[
\cE \colon
\left\{
\begin{tabular}{c}
infinitesimal graded Higgs bundles\\
compatible with $\pi \colon \cX \to S$
\end{tabular}\right\}
\too
\left\{
\begin{tabular}{c}
coherent
sheaves on $B$
\end{tabular} 
\right\}, \quad H \longmapsto \cE(H).
\]
\end{definition}

When  $H = H_{\cL}$ is the fiber at $s$ of the Higgs bundle $\gr \cV$ we have $\gr^p \psi_{H_\cL} = 0$ for $p \neq d$.
Hence in this case we will write
\[ \psi_{H_\cL} = \gr^d \psi_{H_\cL}.\]
To state the reconstruction result, we need to consider for any vector bundle $\cL$ on $X$ the evaluation map
\[ \eta_\cL \colon \rH^0(X, \cL \otimes \omega_X) \otimes \cO_X \to \cL \otimes \omega_X.\]
For a morphism $f\in \Hom(H_1, H_2)$ of infinitesimal graded Higgs bundles, let $\gr^i f\in \Hom(H_1^i, H_2^i)$ denote its component in degree $i\in \bbZ$. We then have:

\begin{proposition} \label{reconstruction-thm}
For every vector bundle $\cL$ on $\cX$ with an integrable unitary connection, we have a morphism 
$ \alpha_\cL \colon \im (\eta_\cL) \otimes \omega_X^\vee \rvert_B \to \cE(H_\cL)$ of $\cO_B$-modules
which is functorial in the flat bundle~$\cL$. Moreover, we have:
\begin{enumerate}
\item If the evaluation maps $\eta_\cL$ and $\eta_{\cL^\vee}$ are surjective on $B$, then $\alpha_\cL$ is an isomorphism 
\[\alpha_\cL \colon  \cL \rvert_B \stackrel{\sim}{\too} \cE(H_\cL).\]
\item Suppose that $B$ contains a divisor $B' \subset X$ such that $\omega^\vee_X(B')$ is ample. If $\eta_{\cL^\vee}$ is surjective on $B$, 
%and no nonzero global section of $\cL \otimes \omega_X$ on $X$ vanishes identically on $B$,
 then for any morphism $f \colon H \to H_\cL$ of infinitesimal graded Higgs bundle compatible with $\pi$, we have
\[ \cE(f) = 0 \implies \gr^d f = 0.\]
\end{enumerate}
\end{proposition}

\begin{proof} 
By construction, the map $c$ is the restriction of the cup product map
\[ \rH^d(X, \omega_X^\vee) \to \Hom(\rH^0(X, \cL \otimes \omega_X), \rH^d(X, \cL))\]
to the image of $\mu$. Note that via Serre duality the adjoint to the previous map can be seen as the map 
$ \rH^0(X, \cL \otimes \omega_X) \to \Hom(\rH^0(X, \cL^\vee \otimes \omega_X), \rH^0(X, \omega_X^{\otimes 2}))$
given by multiplication of global sections. It follows that on $B$ the following diagram is commutative:
\[ 
\begin{tikzcd}[column sep=45pt, row sep=30pt]
\rH^0(X, \cL \otimes \omega_X) \otimes \omega_X^\vee \ar[d, "\eta_\cL \otimes \id", swap] \ar[r, "\psi_{H_\cL}"] & \rH^0(X, \cL^\vee \otimes \omega_X)^\vee \otimes \omega_X \\
\cL \ar[r, equal]& \cHom(\cL^\vee \otimes \omega_X, \omega_X) \ar[u, "(\eta_{\cL^\vee})^\vee \otimes \id", swap]
\end{tikzcd}
\]
The above factorization of $\psi_{H_\cL}$ induces the sought-for functorial morphism $\alpha_\cL$.

(1) Since $\eta_\cL$ and $\eta_{\cL^\vee}$ are assumed to be surjective on $B$, the left vertical arrow in the above diagram is surjective on $B$ and the right vertical arrow is injective on $B$. Hence the claim follows.

(2) The commutativity on $B$ of the diagram
\[ 
\begin{tikzcd}[column sep=45pt]
H^d  \otimes \omega_X^\vee  \ar[d, "{\gr^d f} \otimes \id", swap] \ar[r, "\gr^d \psi_H"]& H^0 \otimes \omega_X  \ar[d, "{\gr^0 f} \otimes \id"]\\
\rH^0(X, \cL \otimes \omega_X) \otimes \omega_X^\vee \ar[d, "{\eta_\cL} \otimes \id", swap] \ar[r, "\psi_{H_\cL}"] & \rH^0(X, \cL^\vee \otimes \omega_X)^\vee \otimes \omega_X \\
\cL \ar[r, equal]& \cHom(\cL^\vee \otimes \omega_X, \omega_X) \ar[u, "(\eta_{\cL^\vee})^\vee \otimes \id", swap]
\end{tikzcd}
\]
together with the fact that $\eta_{\cL^\vee}$ is surjective on $B$, gives that the vanishing of $\cE(f)$ implies that of the composite of the leftmost vertical arrows in the above diagram. It follows that $\gr^d f$ takes its values in sections $s \in \rH^0(X, \cL \otimes \omega_X)$ vanishing identically on $B$ hence on a divisor $B' \subset X$ as in the statement. But the assumptions imply that any such section vanishes identically on $X$: indeed,
\[ \rH^0(X, \cL \otimes \omega_X(- B')) = 0,\]
because $\mathscr{L}\otimes \omega_X(- B')$ is anti-ample, since $\omega_X^\vee(B')$ is ample and $\cL$ is a flat bundle.
\end{proof}

\begin{remark}\label{ampleness-base-locus} The existence of a divisor $B' \subset X$ contained in $B$ such that $\omega^\vee_X(B')$ is ample is trivially verified if $\mu$ is surjective, that is, $B = X$. It is also the case when $\mu$ has corank $1$ and $X$ is canonically polarized: in this case, $B' = B$ is a bicanonical divisor and $\omega^\vee_X(B') \iso \omega_X$ is ample.
\end{remark}

Note that a complex variation of Hodge structures is not determined uniquely by the underlying local system. But the above constructions essentially only depend on the local system $\bbV$:

\begin{lemma}\label{lem:independence-of-choices} 
Let $(H_1, \theta_1)$ and $(H_2, \theta_2)$ be the infinitesimal graded Higgs bundles attached to two different choices of a complex variation of Hodge structures on a given underlying local system $\bbV$ on $S$. 
\begin{enumerate} 
\item If $(H_1, \theta_1)$ is compatible with~$\pi$, then so is $(H_2, \theta_2)$.\smallskip
\item In this case we have a canonical isomorphism 
$\cE(H_1) \simeq \cE(H_2)$.
\end{enumerate}
Hence in this situation we will also denote the coherent sheaf in (2) by $\cE(\bbV):=\cE(H_1)$.
\end{lemma} 

\begin{proof}
Since $\mathbb{V}$ is semisimple, it admits an isotypic decomposition $\mathbb{V}=\bigoplus_i \mathbb{V}_i\otimes W_i$ where the $\bbV_i$ are pairwise non-isomorphic simple local systems $\bbV_i$ and the $W_i$ are vector spaces. By prop.~4.3.13 and rem.~4.3.14 in \cite{MHMProject}, each $\mathbb{V}_i$ underlies a complex variation of Hodge structure which is unique up to a shift of the Hodge bidegrees, and once we fix a choice of these shifts, any complex variation of Hodge structures on $\bbV$ is obtained by choosing a bigrading on each $W_i$. The associated infinitesimal graded Higgs bundle is compatible with $\pi$ iff the one associated with each $\bbV_i$ is, hence claim (1) follows. Moreover, for $\alpha = 1,2$ the isotypic decomposition of the local system induces a canonical isomorphism
\[
\cE(H_\alpha) \;\simeq\; \bigoplus_i \mathscr{E}(\bbV_i)\otimes W_i 
\]
where the right hand side does not depend on the chosen bigrading on $W_i$, hence (2) follows.
\end{proof}

The vanishing statement in \cref{reconstruction-thm} (2) can be put in a more symmetric form, using that the functor $\cE$ is compatible with duality in the following sense. Given an infinitesimal graded Higgs bundle $H$ compatible with $\pi$, we have a nondegenerate pairing
\[\beta_H \colon \cE(H) \times \cE(H^\vee) \rightarrow \cO_B,  \]
defined for all local sections $x$ of $H^p \otimes \omega_X^\vee \rvert_B$ and $y$ of $(H^{d-p})^\vee \otimes \omega_X^\vee \rvert_B$ by
\[ \beta_H(\psi_H(x),\psi_{H^\vee}(y)) =  \langle \psi_{H}(x), y\rangle_H = \langle x, \psi_{H^\vee}(y)\rangle_{H^\vee}, \]
where $\langle -, - \rangle_H$ is the duality pairing between the target of $\psi_{H}$ and the source of $\psi_{H^\vee}$. For any morphism $f \colon H \to H'$ of infinitesimal graded Higgs bundle compatible with $\pi$, we have
\begin{equation} \label{adjoint-reconstruction}
 \beta_{H'}(\cE(f)(x), y) \;=\; \beta_H(x, \cE(f^\vee)(y)).
\end{equation}

\begin{corollary} \label{decomposition-of-triples} 
Suppose that $B$ contains a divisor $B' \subset X$ such that $\omega^\vee_X(B')$ is ample. Let $\cL$ be a vector bundle on $\cX$ with an integrable unitary connection. Suppose that $\eta_{\cL}$ and $\eta_{\cL^\vee}$ are surjective on $B$. 
Then, 
\begin{enumerate}
\item for any endomorphism $f\in \End(H_\cL)$ we have
\[ \cE(f) = 0 \implies \gr^0 f = \gr^d f = 0.\]
\item if $\cL$ is a line bundle and $\rH^0(B, \cO_B) = \bbC$, for any direct sum decomposition $H_\cL \simeq H_1 \oplus \cdots \oplus H_r$ as infinitesimal graded Higgs bundles, there is $i \in \{1, \dots, r\}$ such that both for $p = 0$ and for $p=d$, we have
\[
H_j^p = 
\begin{cases}
H^p_\cL & \text{if $j = i$},\\
0 & \text{otherwise}.
\end{cases}
\]
\end{enumerate}
\end{corollary}

\begin{proof} 
(1) By \cref{reconstruction-thm} (2), it suffices to prove that $\gr^0 f = 0$. From~\cref{adjoint-reconstruction} we know that $\cE(f^\vee)$ is adjoint to $\cE(f)$ and this latter vanishes, we have that $\cE(f^\vee)$ also vanishes by nondegeneracy of~$\beta_{H_\cL}$. The result then follows from \cref{reconstruction-thm} (2) applied to $f^\vee$.

(2) Since $\cL$ is a line bundle and $\rH^0(B, \cO_B) = \bbC$, the only endomorphisms of $\cE(H_\cL) \simeq \cL \rvert_B$ are homotheties with ratio in $\bbC$. Identify $H_i$ with a subobject of $H := H_\cL$ and let $f_i \colon H \to H$ be the projection onto $H_i$ followed by the inclusion of $H_i$ in $H$. The morphism $f_i$ is idempotent, hence so is $\cE(f_i)$. It follows that $\cE(f_i) = \lambda_i \id$ with $\lambda_i \in \{ 0, 1 \}$, so statement (1) implies
\[ \gr^0(f_i) = \gr^d(f_i) = \lambda_i \id.\]
Since $f_1 + \cdots + f_r = \id$ we must have $\lambda_i = 0$ for all but one $i$.
\end{proof}

\begin{corollary} \label{cor:reconstruction-conservative} 
Suppose that $B$ contains an ample divisor $B' \subset X$ such that $\omega^\vee_X(B')$ is ample. For $i = 1, 2$ let $\cL_i$ be a line bundle on $\cX$ with an integrable unitary connection $ \nabla_i$ such that  $\eta_{\cL_i\oplus \cL_i^\vee}$ is surjective on $B$. If the infinitesimal graded Higgs bundles $H_{\cL_1}$ and $H_{\cL_2}$ are isomorphic, then so are the flat bundles $(\cL_1, \nabla_1) \rvert_X$ and $(\cL_2, \nabla_2) \rvert_X$.
\end{corollary}

\begin{proof} Since the connections on $\cL_1$ and $\cL_2$ are unitary, it is enough to show that the line bundles $\cL_1 \rvert_X$ and $\cL_2 \rvert_X$ are isomorphic. \Cref{reconstruction-thm} (1) implies that the lines bundles $\cL_1 \rvert_B$ and $\cL_2 \rvert_B$ are isomorphic. If $B = X$, then there is nothing to prove. If $B$ contains an ample divisor, the result follows from \cref{injectivity-pic0-restriction-ample-divisor} below.
\end{proof}

\begin{lemma} \label{injectivity-pic0-restriction-ample-divisor} For any ample divisor  $D$ in a smooth projective variety $Y$ of dimension at least $2$, the natural map $\Pic^0(Y) \to \Pic(D)$
is injective.
\end{lemma}

\begin{proof} Let $\cL$ be an algebraically trivial line bundle on $Y$ and suppose that its restriction to $D$ is trivial, i.e. it admits a nowhere vanishing section $\sigma \in \rH^0(D, \cL)$. Taking global sections of the short exact sequence $0 \to \cL(-D) \to \cL \to \cL \rvert_D \to 0$ yields a short exact sequence
\[ 0 \too \rH^0(Y, \cL(-D)) \too \rH^0(Y, \cL) \too \rH^0(D, \cL) \too 0\]
because $\rH^1(Y, \cL(-D)) = 0$ as $\cL(-D)$ is antiample, by the Kodaira vanishing theorem. The section $\sigma$ then lifts to a nonzero section of $\cL$ on $Y$, yielding an isomorphism $\cO_Y \iso \cL$ since $\cL$ is algebraically trivial.
\end{proof}

\subsection{Consequences for monodromy} 
Let $\cL$ be a line bundle on $\cX$ with an integrable connection~$\nabla$ and let $\bbL = \ker \nabla$ be the local system on $\cX(\bbC)$ of its parallel sections. We consider the local system
\[ \bbV \; := \; \rR^d \pi_\ast \bbL. \]
Suppose moreover that $\cX \to S$ is projective and that the local system $\bbL$ has unitary monodromy. In this case $\bbV$ underlies a natural polarizable complex variation of Hodge structures, and in particular admits a Hodge decomposition, for any $t \in S$,
\[ \bbV_t \; = \; \bigoplus_{p + q = d} \bbV_t^{p,q} \quad \text{with} \quad \bbV_t^{p, q} \simeq \rH^q (\cX_t, \cL \otimes \Omega^p_{\cX / S}).\]

\begin{theorem} \label{th:consequences-monodromy} 
Suppose that $B$ contains a divisor $B' \subset X$ such that $\omega^\vee_X(B')$ is ample, and $\rH^0(B, \cO_B) = \bbC$.
If $\eta_{\cL}$ and $\eta_{\cL^\vee}$ are
surjective on $B$, 
then there is a unique simple complex local system $\bbW \subset \bbV = \rR^d \pi_\ast \bbL$ such that
\[ \bbV^{d,0}_t \oplus \bbV^{0,d}_t \; \subset \; \bbW_t \qquad \text{for all $t \in S$}.\]
If $d\geq 1$ and if $h^0(X,\mathscr{L}\otimes  \omega_X)$ and $h^d(X,\mathscr{L})$ are both non-zero, we have 
\[ \rk (\bbW) \ge h^0(X,\mathscr{L}\otimes  \omega_X)+h^d(X,\mathscr{L}) + d-1.\]
\end{theorem}

\begin{proof} By proposition 4.3.13 and remark 4.3.14 in \cite{MHMProject}, the local system $\bbV$ decomposes as a direct sum 
\[ \bbV \; = \; \bigoplus_{i = 1}^r \bbV_i\]
where each $\bbV_i \subset \bbV$ is a simple local subsystem underlying a subvariation of complex Hodge structures. 
By compatibility with the Hodge decomposition, it follows that there is an index $i_0$ such that the local subsystem $\bbW := \bbV_{i_0} \subset \bbV$ satisfies $\bbW_s \cap \bbV^{d,0}_s \neq 0$ for some $i$. Thus \cref{decomposition-of-triples} (2) implies
\[ \bbV^{d, 0}_s \oplus \bbV^{0,d}_s \subset \bbW_s. \]
Since $\bbW \subset \bbV$ underlies a subvariation of complex Hodge structures, this inclusion at $s$ implies the same inclusion at every $t\in S$ and gives
\begin{eqnarray*}
 \dim(\bbW_t) 
 &\ge & \dim(\bbV^{0,d}_t) + \dim(\bbV^{d,0}_t)+\sum_{i=1}^{d-1} \dim \mathbb{W}_t^{i,d-i} 
 \\
 &=& h^0(X,\mathscr{L}\otimes  \omega_X)+h^d(X,\mathscr{L})+\sum_{i=1}^{d-1} \dim \mathbb{W}_t^{i,d-i}.
\end{eqnarray*}
It remains to show that if $h^0(X,\mathscr{L}\otimes  \omega_X)$ and $h^d(X,\mathscr{L})$ are both non-zero, then so are all the summands on the right hand side. This follows from Griffiths transversality: if $\mathbb{W}^{p,d-p}_t = 0$ for some $p\in \{1, \dots, d-1\}$, then
\[
 \cF^p = \cF^{p+1}
 \quad \text{for the Hodge filtration $\cF^\bullet$ on $\cW = \bbW\otimes \cO_S$}.
\]
Then Griffiths transversality implies $\nabla(\cF^p) \subset \cF^p \otimes \Omega^1_S$, which means that $\cF^p \subset \cW$ is a flat subbundle. Since $\bbW$ is a simple local system, it follows that $\cF^p = 0$ or $\cF^p = \cW$. But since by construction
\[
\cF^p_s = \mathbb{V}^{d,0}_s\oplus \bigoplus_{i=1}^{p-1} \mathbb{W}_s^{d-i, i},
\]
we know that $\cF^p \neq 0$ (as $\mathbb{V}^{d,0}_s \neq 0$) and that $\cF^p \neq \cW$ (as $\mathbb{V}^{0,d}_s \neq 0$). This yields the desired contradiction.
\end{proof}

Suppose now that the flat bundle $\cL$ is torsion. By its \emph{order} we mean the least integer $n \ge 1$ such that $\cL^{\otimes n}$ is the trivial flat bundle. Let $K := \bbQ(\zeta_n) \subset \bbC$ be the $n$-th cyclotomic extension and let~$R:= \cO_K$ be its ring of integers. Then $\bbL$ comes by extension of scalars from a local system $\bbL_R$ with coefficients in~$R$. Hence $\bbV$ is obtained from the local system $\bbV_R = \rR^d \pi_\ast \bbL_R$ with coefficients in $R$ by extension of scalars:
\[ \bbV = \bbV_R \otimes_{R} \bbC.\]
In what follows we will be interested in local subsystems of~$\bbV_R$ with coefficients in $\bbZ$. We consider the tensor product
\begin{equation} \label{galois-decomposition}
 \bbV_R \otimes_\bbZ \bbC = \bigoplus_{i \in (\bbZ / n \bbZ)^\times} \bbV_i \quad \textup{where} \quad \bbV_i := \rR^d \pi_\ast \bbL^{\otimes i}.
 \end{equation}
 Needless to say, with this notation we have $\bbV_1 = \bbV$. If the hypotheses of \cref{th:consequences-monodromy} hold for $B$ and the flat line bundle $\cL^{\otimes i}$, we may consider the unique complex sublocal system
\[
\bbW_i \subset \bbV_i
\]
provided by \cref{th:consequences-monodromy}.

\begin{proposition} \label{Cor:summand-finite-monodromy} 
Let $d=2$. Suppose that $B$ contains a divisor $B' \subset X$ such that $\omega^\vee_X(B')$ is ample, and $\rH^0(B, \cO_B) = \bbC$. Let $\cL$ be a torsion line bundle of order $n$ such that for all $i \in (\bbZ / n \bbZ)^\times$ the maps $\eta_{\cL^{\otimes i}}$ are surjective on $B$. Then for any $R$-local subsystem 
\[ \bbV'_R \subset \bbV_R \]
such that $\bbV'_R \otimes_\bbZ \bbC \subset \bbV_R \otimes_\bbZ \bbC$ meets all $\bbW_i$ trivially, the monodromy of $\bbV'_R$ is finite.
\end{proposition}

\begin{proof} Since $\bbV'_R$ is an $R$-sublocal system, the complex local system $\bbV'_R \otimes_{\bbZ} \bbC$ admits a decomposition compatible with the one in~\eqref{galois-decomposition}:
\[ \bbV'_R \otimes_\bbZ \bbC = \bigoplus_{i \in (\bbZ / n \bbZ)^\times} \bbV'_i \quad \textup{where} \quad \bbV'_i = (\bbV'_R \otimes_\bbZ \bbC) \cap \bbV_i.\]
By assumption, the local systems $\bbV'_i$ and $\bbW_i$ meet trivially, hence $\bbV'_i$ must be pure of Hodge type~$(1, 1)$ because $d = 2$. In particular, $\bbV'_i$ is unitary. Seeing $\bbV'_R$ just as an integral local system, the local system $\bbV_R'$ is then both integral and unitary, hence it has finite monodromy.
\end{proof}

\begin{corollary} \label{supplement-finite-monodromy} 
Let $d=2$. Suppose that $B$ contains a divisor $B' \subset X$ such that $\omega^\vee_X(B')$ is ample, and $\rH^0(B, \cO_B) = \bbC$. Let $\cL$ be a torsion line bundle of order $n$ such that for all $i \in (\bbZ / n \bbZ)^\times$ the maps $\eta_{\cL^{\otimes i}}$ are surjective on $B$. If every simple sublocal system of $\bbV_i$ different from $\bbW_i$ has rank $< \rk \bbW_i$, then
\[ \bbV = \bbW \oplus \bbW'\]
where the local system $\bbW'$ has finite monodromy.
\end{corollary}

\begin{proof} Under the assumptions on the rank of $\bbW$, the simple direct summand $\bbW$ is an isotypic component of~$\bbV$, hence it admits a unique complement $\bbW'$. The point is that in this case $\bbW'$ comes from a $R$-local subsystem of $\bbV_R$ because $\bbW$ does. To argue for this last claim, notice that $\bbL_R^{\otimes i}$ can also be seen as the Galois conjugate $\bbL_R^\sigma$ for the unique $\sigma \in \Gal(K / \bbQ)$ such that $\sigma(\zeta_n) = \zeta_n^i$. As a consequence, 
\[ \bbV_{i, R} := \rR^2 \pi_\ast \bbL_R^{\otimes i} = \bbV_R^\sigma\]
is the Galois conjugate of $\bbV_R$ under $\sigma$. Now, any local complex subsystem of $\bbV_i$ is defined over~$\overline{\bbQ}$, thus it makes sense to consider its Galois conjugates under $\Gal(\overline{\bbQ} / \bbQ)$. Since the simple sublocal systems of $\bbV_i$ different from $\bbW_i$ have smaller rank, the sublocal system $\bbW_i$ coincides with its Galois conjugates under $\Gal(\overline{\bbQ} / K)$, hence it comes by extension of scalars from a unique saturated $R$-sublocal system
\[ \bbW_{i, R} \subset \bbV_{i, R}.\]
The same uniqueness property shows
\[ \bbW_{i, R} = \bbW_{R}^\sigma\]
where $\bbW_R := \bbW_{1, R}$ and $\sigma \in \Gal(K / \bbQ)$ is such that $\sigma(\zeta_n) = \zeta_n^i$. It follows at once that the complement $\bbW' \subset \bbV$ of $\bbW$ comes from local subsystem $\bbW'_R \subset \bbV_R$ with coefficients in $R$. Moreover, seen as a $\bbZ$-local subsystem it is such that $\bbW'_R \otimes_\bbZ \bbC$ meets trivially $\bbW_{i, R}$ for any $i \in (\bbZ / n \bbZ)^\times$. \Cref{Cor:summand-finite-monodromy} implies that $\bbW'$ has finite monodromy, as desired.
\end{proof}

%%%%%%%%%%%%%%%%%%%%%%%%%%%%%%%%%%%
%
% THE FANO SURFACE
%
%%%%%%%%%%%%%%%%%%%%%%%%%%%%%%%%%%%

\section{The surface of lines on a cubic threefold}\label{sec:fano}

\subsection{Main results} \label{sec:StatementsCubic3fold}

To prove the main theorem of this paper, we are going to apply the results of the previous section to the family of Fano surfaces of lines deduced from a versal family of cubic threefolds. In this section we are going to prove that the hypotheses of \cref{th:consequences-monodromy} are fulfilled for these surfaces. More precisely, let $Y \subset \bbP(E)$ be a smooth cubic threefold, where $\dim E  = 5$. Let $F:= F(Y)$ be the surface of lines on $Y$. First of all, we show the following global generation result for algebraically trivial line bundles:

\begin{theorem} \label{thm:global-generation-line-bundle}For $\cL \in \Pic^0(F)$ generic, the line bundle $\cL \otimes \omega_F$ is globally generated. In particular, if $\Pic^0(F)$ is simple, $\cL \otimes \omega_F$ is globally generated for all but finitely many torsion line bundles $\cL$.
\end{theorem}

The second result concerns the rank of the linear map $\mu \colon \Sym^2 \rH^1(F, T_F) \to \rH^2(F, \textstyle \bigwedge^2 T_F)$
induced by the cup product. The geometric meaning of this rank is better understood by considering the dual map
\[ \mu^\vee \colon \rH^0(F, \omega_F^{\otimes 2}) \to \Sym^2 \rH^1(F, T_F)^\vee\]
where we use the identification $\rH^2(F, \textstyle \bigwedge^2 T_F)^\vee \iso \rH^0(F, \omega_F^{\otimes 2})$ given by Serre duality. By definition, the surface $F$ embeds into the Grassmannian of projective lines in $\bbP(E)$,
\[ F \into G:= \Gr_2(E).\]
The polarization $\cO_G(1)$ on $G$ given by the Pl\"ucker embedding restricts to the canonical bundle of $F$, and the restriction map
\[ \rH^0(G, \cO_G(2)) \into \rH^0(F, \omega_F^{\otimes 2})\]
is injective. However it is not an isomorphism as cor. 1.8 and prop. 1.15 in \cite{AltmanKleiman} yield
\[ h^0(G, \cO_G(2)) = 50 \qquad \text{and} \qquad h^0(F, \omega_F^{\otimes 2}) = 51. \]
Recall that a line $L \subset Y$ is of \emph{second type} if its normal bundle is
$ \cN_{L / Y} \simeq \cO_L(1) \oplus \cO_L(-1)$. Lines of second type form a divisor in $F$ cut out by a bicanonical section \[\sigma \in \rH^0(F, \omega_F^{\otimes 2}).\]
With this notation, we have:

\begin{theorem} \label{thm:rank-of-cup-product} We have $\mu^\vee(\sigma) = 0$ and, for $Y$ generic, the composite map
\[\rH^0(G, \cO_G(2)) \into \rH^0(F, \omega_F^{\otimes 2}) \stackrel{\mu^\vee}{\too} \Sym^2 \rH^1(F, T_F)^\vee\]
is injective. In particular, for $Y$ generic, $\rk \mu = 50$, and the base locus $B$ of $\ker(\mu^\vee) \subset \rH^0(F, \omega_F^{\otimes 2})$ is the divisor of lines of second type on $Y$.
\end{theorem}

As usual, in the above statement, the generic quantifier means that the statement holds for $Y$ in a suitable nonempty open subset of $\bbP(\Sym^3 E^\vee)$. Also, $B$ is a bicanonical divisor, hence ample. 

The rest of this section is devoted to the proof of these two results.

\subsection{Proof of \cref{thm:global-generation-line-bundle}}  We begin with the following precise form of generic vanishing for $F$:

\begin{lemma}\label{lem:line-bundles-coh}
For any nontrivial $\cL\in \Pic^0(F)$ we have 
\[h^0(F, \cL)=h^1(F, \cL)=0 \quad \text{and} \quad h^2(F, \cL)=6.\]
\end{lemma}
\begin{proof}
Since $\rH^0(F,\cL)=0$ for any non-trivial line bundle $\cL$ with vanishing first Chern class and since
$\chi(F, \cL)=\chi(F, \mathscr{O}_F)=6$, we only need to see \[\rH^1(F, \cL)=0.\] For this we will use Hodge theory. Since $\cL$ has vanishing first Chern class, it admits a (unique) unitary flat connection $\nabla \colon \cL\to \cL\otimes \Omega^1_F$. Let $\mathbb{L} = \ker \nabla$ be the local system of its flat sections. Because of the Hodge decomposition
\[\rH^1(F, \mathbb{L})=\rH^1(F, \cL)\oplus \rH^0(F, \cL\otimes \Omega^1_F)\] it suffices to show $\rH^1(F, \mathbb{L})=0$ for every non-trivial rank one local system $\mathbb{L}$ on $F$. In other words, we want to show
$$\rH^1(\pi_1(F), \chi)=0$$
for the character $\chi\colon  \pi_1(F)\to \mathbb{C}^\times$ corresponding to $\mathbb{L}$. Collino \cite{collino12, collino2012remarks} showed that the topological fundamental group $\pi_1(F)$ is a nontrivial central extension of $\bbZ^{10}$ by $\bbZ/2\bbZ$. By the Hochschild--Serre spectral sequence, this gives us an exact sequence $$0\to \rH^1(\bbZ^{10}, \chi)\to \rH^1(\pi_1(F), \chi)\to \rH^1(\bbZ/2, \mathbb{C}).$$ The left term in this sequence vanishes because $\chi$ is non-trivial, and the right vanishes because $\mathbb{Z}/2\mathbb{Z}$ is torsion. Hence the same is true for the central term.
\end{proof}

We now prove the global generation result in \cref{thm:global-generation-line-bundle}. Let $S \subset \Pic^0(F)$ be a smooth curve passing through $0$ with tangent direction
\[ v \in T_0 \Pic^0(F) \simeq \rH^1(F, \cO_F).\]
Let $\cL$ be the restriction of the Poincar\'e bundle to $F \times S$ and let
\[ \pr_F \colon F \times S \to F, \qquad \pr_S \colon F \times S \to S,\]
be the two projections. The coherent sheaf
\[ \cE := \pr_{S \ast} (\cL \otimes \pr_F^\ast \omega_F)\]
on $S$ is torsion free, thus locally free because $S$ is a smooth curve. Moreover, by the preceding lemma and the base change theorem for coherent cohomology, the vector bundle $\cE$ has rank $6$ and, for any $s \in S\smallsetminus \{0\}$, it has  fiber $\rH^0(F, \cL_s \otimes \omega_F)$ where $\cL_s$ is the restriction of $\cL$ to the fiber of $\pr_S$ at~$s$. The locus in $F\times S$ where the evaluation morphism 
\[ \pr_S^\ast \cE = \pr_S^\ast \pr_{S \ast} (\cL \otimes \pr_F^\ast \omega_F) \to \cL \otimes \pr_F^\ast \omega_F\]
is surjective is open, and its complement has a closed image $T \subset S$ via $\pr_S$ by properness of $F$. Therefore it suffices to prove that $0 \not\in T$. To do  this, we are going to prove that the tangent vector $v$ can be chosen so that the subspace
\[ L_v:= \cE_0 \into \rH^0(F, \omega_F)\]
defines a base-point free linear system. We begin by giving an alternative description of $L_v$. To state it, recall that the cup product map $\rH^1(F, \cO_F) \otimes \rH^1(F, \cO_F) \to \rH^2(F, \cO_F)$ is skew-symmetric and the induced linear map
\begin{equation} \label{H2-is-wedge-square} \textstyle \bigwedge^2 \rH^1(F, \cO_F) \stackrel{\sim}{\too} \rH^2(F, \cO_F) \end{equation}
is an isomorphism by~\cite[chapter 5, lemma 2.5]{huybrechts2023geometry}.

\begin{lemma} \label{lemma:kernel-contraction} 
Via the identification $\rH^0(F, \omega_F) \simeq \bigwedge^2 \rH^1(F, \cO_F)^\vee$ given by Serre duality and \eqref{H2-is-wedge-square}, we have
\[ L_v = \{ \textup{skew-symmetric bilinear forms $\phi$ on $\rH^1(F, \cO_F)$ with $v \in \ker \phi$} \}. \]
\end{lemma}

\begin{proof}[{Proof of \cref{lemma:kernel-contraction}}] Standard arguments in deformation theory \cite[Proposition 3.3.4]{sernesi2006deformations} imply that~$L_v$ lies in the kernel of the linear map $\rH^0(F, \omega_F) \to \rH^1(F, \omega_F)$ given by cup product with $v$:
\[ L_v \; \subset \; L'_v := \ker(\rH^0(F, \omega_F) \to \rH^1(F, \omega_F) : \alpha \mapsto \alpha \cup v).\]
On the other hand, applying Serre duality and taking the adjoint map of the cup product map
\[ \rH^0(F, \omega_F) \otimes \rH^1(F, \cO_F) \to \rH^1(F, \omega_F)\]
one obtains the cup product map $\rH^1(F, \cO_F) \otimes \rH^1(F, \cO_F) \to \rH^2(F, \cO_F)$. Via the isomorphism \eqref{H2-is-wedge-square} the kernel $L'_v$ can thus be seen as the subspace of skew-symmetric bilinear forms $\phi$ on $\rH^1(F, \cO_F)$ such that $\phi(v, -)$ vanishes identically. These are in linear bijection with skew-symmetric bilinear forms on $\rH^1(F, \cO_F) / \bbC v$, thus 
\[ \dim L'_v = \textstyle \binom{5 - 1}{2} =  6.\]
On the other hand, the subspace $L_v$ has dimension $6$ because it is the fiber at $0$ of the rank $6$ vector bundle $\cE$. Therefore $L_v = L'_v$, which concludes the proof.
\end{proof}

To analyze the base locus of $L_v$, let $A:=\Pic^0(F)$.
The cubic hypersurface $Y \subset \bbP^4$ can be reconstructed from $F$ as the image of the map
\[ \bbP(T_F) \into \bbP(T_{A}) = \bbP(\Lie A) \times A \to \bbP(\Lie A) = \bbP(\rH^1(F, \cO_F))\] (see \cite[Chapter 5, proof of theorem 4.3]{huybrechts2023geometry})
after taking an isomorphism $\bbP(\rH^1(F, \cO_F)) \simeq \bbP(E)$ as follows. 
By \cite[Chapter 2, Corollary 4.20(ii)]{huybrechts2023geometry} there is a canonical isomorphism
\[ \rH^1(F, \cO_F) \simeq \rH^{1,2}(Y) = \rH^2(Y, \Omega^1_Y).\]
On the other hand, $\rH^{1,2}(Y)$ is the degree $4$ component of the Jacobian ring $R$ of $Y$ and the natural pairing $R_1 \times R_4 \to R_5$ is perfect and gives an isomorphism $R_4 \simeq R_1^\vee \otimes R_5$. Now $R_1 = E^\vee$ and $R_5$ is of dimensional one by smoothness of $Y$, so the isomorphism
\begin{equation} \label{eq:choice-coordinates}\rH^1(F, \cO_F) \iso E \otimes R_5\end{equation}
gives the wanted isomorphism $\bbP(\rH^1(F, \cO_F)) \iso \bbP(E)$, through which we see $[v] \in \bbP(\rH^1(F, \cO_F))$ as a point of $\bbP(E)$.

\begin{lemma} A point $x \in F$ is a base point of the linear system $L_v \subset \rH^0(F, \omega_F)$ if and only if the corresponding line $\ell_x \subset Y$ contains $[v] \in \bbP(E)$.
\end{lemma}

\begin{proof} Pick a nonzero element of $R_5$ and identify $\rH^1(F, \cO_F)$ with $E$ via \eqref{eq:choice-coordinates}. In the above description of $Y$, the line $\ell_x$ is the subspace $\bbP(T_x F) \subset \bbP(\rH^1(F, \cO_F))$. Moreover, evaluating a global differential $2$-form at $x$ corresponds to seeing it as skew-symmetric bilinear form on $\rH^1(F, \cO_F)$ via the isomorphism $\rH^0(F, \omega_F) \iso \bigwedge^2 \rH^1(F, \cO_F)^\vee$ and then restricting such a form to $T_x F$. Since $L_v$ is the subspace of skew-symmetric forms $\phi$ with $v \in \ker \phi$, saying that any such form vanishes on~$T_x F$ means that $v$ belongs to $T_x F$. 
\end{proof}

In particular, the linear system $L_v$ is base-point free if and only if $[v] \notin Y$. Therefore, the line bundle $\mathscr{L}\otimes \omega_F$ is globally generated for general $\mathscr{L}\in \on{Pic}^0(F)$. 

To conclude the proof of \cref{thm:global-generation-line-bundle}, it remains to justify the final claim.
By assumption the abelian variety $A = \Pic^0(F)$ is simple. Let $\cP$ be the Poincar\'e bundle on $F \times A$ and 
\[ \pr_F \colon F \times A \to F, \qquad \pr_A \colon F \times A \to A,\]
 the two projections. By the \cref{lem:line-bundles-coh} and the base change theorem for coherent cohomology, the coherent sheaf $\pr_{A \ast} (\cP \otimes \pr_F^\ast \omega_F)$
has fiber $\rH^0(F, \cM \otimes \omega_F)$ at any nontrivial $\cM \in A$. Again, the locus in $F\times A$ where the morphism 
\[ \pr_A^\ast \pr_{A \ast} (\cP \otimes \pr_F^\ast \omega_F) \to \cP \otimes \pr_F^\ast \omega_F\]
is surjective is open, and its complement has a closed image $Z \subset A$ via $\pr_A$. Moreover, by construction, the intersection of $Z$ with $A \smallsetminus \{ 0 \}$ is the locus of nontrivial $\cM \in A$ such that $\cM \otimes \omega_F$ is not globally generated. Let 
\[ \widetilde{Z} \subset A\] be the Zariski closure of $Z \cap A_{\mathrm{tor}}$, where $A_{\mathrm{tor}} \subset A$ denotes the torsion subgroup. By the main result of \cite{raynaud1983sous}, the subvariety $\widetilde{Z}$ is a finite union of translates of abelian subvarieties of $A$. As we just showed, the subvariety $Z$ is not the whole of $A$. Therefore, if $A$ is simple, then $\widetilde{Z}$ is finite. This completes the proof of \cref{thm:global-generation-line-bundle}. \qed 

\subsection{The Jacobian ring of the Fermat cubic threefold} \label{sec:Fermat-cubic} We now begin preparations to prove \cref{thm:rank-of-cup-product}. To prove the generic lower bound for the rank of the map $\mu$ in \cref{thm:rank-of-cup-product} we will reduce to the case of Fermat cubic threefold
\[
Y\subset \bbP^4
\quad \text{with equation} \quad 
F=x_0^3+x_1^3+x_2^3+x_3^3+x_4^3.
\]
This case is down-to-earth enough to carry out an explicit computation in the Jacobian ring
\[
R:=\bbC[x_0,\dots,x_4]/ J \quad \textup{where} \quad J = (\partial F / \partial x_i : i = 0, \dots, 4) = (x_0^2, \dots, x_4^2).
\]
Let $R_i$ be the $i$-th graded piece of $R$. In the proof of \cref{thm:rank-of-cup-product} we will be interested in computing the rank of the linear map
\[ \nu \colon \Sym^2 R_3 \too \Hom(\textstyle \bigwedge^2 R_1, \bigwedge^2 R_4), \qquad f \cdot g \longmapsto [\phi \wedge \psi \mapsto f \phi \wedge g \psi +  g \phi \wedge  f \psi].  \]
For $a, b \in R$ here we wrote $ab \in R$ for the multiplication of the ring $R$ and $a \cdot b \in \Sym^2 R$ for the multiplication in the symmetric algebra over $R$.

\begin{proposition} \label{Prop:Fermat-rank-50} The linear map $\nu$ has rank $50$. \end{proposition}

\begin{proof} As we said, the proof is an honest computation. To begin with, for a subset $I = \{ i_1, \dots, i_k \}$ of $\{ 0, \dots, 4\}$ write
\[ x_I = x_{i_1 \dots i_k} := x_{i_1} \cdots x_{i_k} \in R_k.\]
With this notation, we have $R_k = \langle x_I : I \subset \{ 0, \dots, 4 \} , |I| = k\rangle$.
To express the map $\nu$ with respect to this basis, notice that, for any $t, u \in \{0, \dots, 4\}$ and any subsets $I, J \subset \{ 0, \dots, 4\}$ of cardinality $3$, we have
\[ x_I x_t \wedge x_J x_u \neq 0 \quad \iff \quad  t \not \in I, \; \; u \not \in J, \; \; I \cup \{t \} \neq J \cup \{ u \}. \]
Moreover, if this is the case and $I \neq J$, then $x_J x_t \wedge x_I x_u = 0$. The explicit expression of $\nu$ then depends on the cardinality of $I \cap J$. If $|I \cap J| = 3$, that is $I = J$, we have
\[
\nu(x_I \cdot x_I)(x_t \wedge x_u) = 
\begin{cases}
2 x_I x_t  \wedge x_I x_u & \textup{if $I \cup \{ t, u\} = \{0, \dots, 4\}$},\\
0 & \text{otherwise}.
\end{cases}
\]
Suppose now $|I \cap J| = 2$. In this case the sets $I \smallsetminus J$, $J \smallsetminus I$ and $\{ 0, \dots, 4\} \smallsetminus (I \cup J)$ are singletons and we call respectively $i$, $j$ and $k$ their elements. Then,
\[
\nu(x_I \cdot x_J)(x_t \wedge x_u) = 
\begin{cases}
\pm \, x_I x_j  \wedge x_J x_k & \textup{if $t = j$, $u = k$,}\\
\pm \, x_I x_k  \wedge x_J x_i & \textup{if $t = k$, $u = i$,}\\
0 & \text{otherwise}.
\end{cases}
\]
Finally, if $|I \cap J| = 1$, we have
\[
\nu(x_I \cdot x_J)(x_t \wedge x_u) = 
\begin{cases}
\pm \, x_I x_t  \wedge x_J x_u & \textup{if $t \not \in I$, $u \not \in J$,} \\
\pm \, x_I x_u  \wedge x_J x_t & \textup{if $u \not \in I$, $t \not \in J$,} \\
0 & \text{otherwise}.
\end{cases}
\]

The other computational advantage of the Fermat cubic threefold is that the map $\nu$ is equivariant under the natural action on $R$ of the subgroup $G \subset \GL_5(\bbC)$ fixing $F$. The group $G$ is the semidirect product of the subgroup $S_5$ permuting coordinates and the subgroup $T = \mu_3^5$ of diagonal matrices whose nonzero entries are third roots of unity. For $k = 1, 2, 3$ the subspace
\[ V_k := \bigoplus_{|I \cap J| = k} \langle x_I \cdot x_J\rangle \subset \Sym^2 R_3 \]
is stable under the action of $G$, where the direct sums ranges over unordered pairs $I, J$ of subsets with cardinality $3$. Moreover, if $|I \cap J| \ge 2$, the subspace $\langle x_I \cdot x_J \rangle$ is an eigenspace for the action of~$T$. Indeed, $x_I \cdot x_J$ is an eigenvector for the character
\[ T \ni t = (t_0, \dots, t_4) \longmapsto \prod_{i = 0}^4 t_i^{\delta_I(i) + \delta_J(i)} \]
where $\delta_I \colon \{0, \dots, 4 \} \to \{0,1\}$ is the indicator function of $I$, and similarly for $J$. When $|I \cap J| \ge 2$, the above character determines the unordered pair $I,J$ uniquely. Instead, when $|I \cap J| = 1$, the eigenspace decomposition of $V_1$ is
\[ V_1 = \bigoplus_{i = 0}^4 W_i \quad \textup{with} \quad W_i = \bigoplus_{I \cap J = \{ i \}} \langle x_I \cdot x_J \rangle.\]
Indeed, if $I \cap J = \{ i \}$, the character of $x_I \cdot x_J$ is $t_i^2 \prod_{j \neq i} t_j$, hence depends only on $i$. Note that, for fixed $i$, there are exactly three unordered pairs $I, J$ of subsets of cardinality $3$ such that $I \cap J = \{ i \}$.

To conclude the proof, since $\nu$ is morphism of representations of $G$, it suffices to compute the rank of the restriction of $\nu$ to each $T$-eigenspace of $\Sym^2 R$. If $|I \cap J| \ge 2$, we saw that $\nu(x_I \cdot x_J)$ is nonzero, hence injective on the eigenspace $\langle x_I \cdot x_J \rangle$. Instead, we claim that the kernel of $\nu\rvert_{W_i}$ has dimension one. This will conclude the proof, since we will then have
\[ \rk \nu = \dim \Sym^2 R_3 - 5 = 50,\]
because $R_3 \iso \rH^1(Y, T_Y)$ has dimension $10$. To prove the claim, we may assume $i =0$, so that a basis of $W_0$ is
\[
f:=x_{012}\cdot x_{034},\qquad
g:=x_{013}\cdot x_{024},\qquad
h:=x_{014}\cdot x_{023}.
\]
The explicit expression of $\nu$ show that it has rank $\ge 2$ on $W_0$, because for example we have:
\begin{align*}
\nu(f)(x_2\wedge x_3) &= x_{0234}\wedge x_{0123},
&\nu(g)(x_2\wedge x_3) &=- x_{0234}\wedge x_{0123},
&\nu(h)(x_2\wedge x_3) &= 0,
\\
\nu(f)(x_2\wedge x_4) &= x_{0234}\wedge x_{0124},
&\nu(g)(x_2\wedge x_4) &= 0, 
&\nu(h)(x_2\wedge x_4) &= - x_{0234}\wedge x_{0124}.
\end{align*}
The above also shows $\nu(f + g + h)(x_2 \wedge x_3) =0$. But the stabilizer of $0$ in $S_5$ acts transitively on unordered pairs in $\{1, \dots, 4\}$ and leaves $f + g + h$ invariant, thus $\nu(f + g + h) = 0$.
\end{proof}

\subsection{The upper bound on the rank in \cref{thm:rank-of-cup-product}}\label{subsec:rank-cup-50} Let $Y \subset \bbP(E)$ be a smooth cubic threefold, where~$E$ is a vector space of dimension $5$, and $F:=F(Y)$ the surface of lines on $Y$. With notation as in \cref{thm:rank-of-cup-product}, we are interested in the linear map
\[ \mu \colon \Sym^2 \rH^1(F, T_F) \to \rH^2(F, \textstyle \bigwedge^2 T_F)\]
induced by the cup product. Via Serre duality, the dual of $\mu$ can be seen as a map
\[ \mu^\vee \colon \rH^0(F, \omega_F^{\otimes 2}) \to \Sym^2 \rH^1(F, T_F)^\vee.\]
We first show that $\mu^\vee$ vanishes on the bicanonical sections cutting out the divisor of lines of second type. There is a distinguished such section, obtained as follows. Let \(p \colon P \to F\) be the universal line and \(q \colon P \to Y\) the projection. The proof of \cite[Chapter 5, Proposition 1.1]{huybrechts2023geometry} shows that the determinant of the differential \(\mathrm{d}q \colon T_P \to q^\ast T_Y\) defines a section of
\[
\mathcal{H}om(\det T_P, q^\ast \det T_Y) \cong p^\ast \omega_F^{\otimes 2}, 
\]
and that via the above identification we have
\[
\det(\rd q) = p^\ast \sigma \quad \text{for a unique \(\sigma \in \rH^0(F, \omega_F^{\otimes 2})\)}.
\]
 By \cite[Chapter 5, Proposition 1.1]{huybrechts2023geometry}, the section \(\sigma\) cuts out the divisor of lines of second type.

\begin{proposition} \label{prop:kernel-cup-product} With the above notation, we have $\mu^\vee( \sigma ) = 0$. In particular, we have $\rk \mu \le 50$.
\end{proposition}

\begin{proof}
Since $\mu$ is the map induced by the cup product, unwinding the identification given by Serre duality, it suffices to show that
\[
\sigma(u \cup v) = 0 \qquad \text{for all } u, v \in \rH^1(F, T_F).
\]
Note that $u \cup v \in \rH^2(F, \bigwedge^2 T_F)$, hence evaluating $\sigma$ on it yields an element of $\rH^2(F, \omega_F)$. The construction of the Fano surface works in families; in particular, it induces a map between the first-order deformations of $Y$ and of $F$. Moreover, by \cite[chapter 5, prop.~2.14]{huybrechts2023geometry}, this map is an isomorphism
\[
\rH^1(Y, T_Y) \stackrel{\sim}{\too} \rH^1(F, T_F),
\]
whose inverse we denote by $u \mapsto u_Y$. Since $P$ is the projective tangent bundle of $F$, we also have a natural linear map between first-order deformations $\rH^1(F, T_F) \to \rH^1(P, T_P)$, written $u \mapsto u_P$. One checks that the following diagram commutes:
\begin{equation} \label{double-square}
\begin{tikzcd}[column sep=35pt]
\rH^1(Y, T_Y) \ar[r, "\sim"] \ar[d, "q^\ast", swap] & \rH^1(F, T_F) \ar[r, equal]\ar[d]& \rH^1(F, T_F) \ar[d, "p^\ast"] \\
\rH^1(P, q^\ast T_Y) & \rH^1(P, T_P) \ar[l, "\rd q", swap]  \ar[r, "\rd p"] & \rH^1(P, p^\ast T_F).
\end{tikzcd}
\end{equation}
To prove the statement, it suffices to show that
\[
(p^\ast \sigma)(p^\ast u \cup p^\ast v) = p^\ast\bigl(\sigma(u \cup v)\bigr)
\]
vanishes. Indeed, the pullback map $p^\ast \colon \rH^2(F, \omega_F) \to \rH^2(P, p^\ast \omega_F)$ is injective by the Leray spectral sequence, since $\rR^1 p_\ast \cO_P = 0$ (as $P$ is a $\bbP^1$-bundle over $F$). Now consider the commutative diagram
\[
\begin{tikzcd}[column sep=35pt]
\rH^1(P, p^\ast T_F)^{\otimes 2} \ar[r, "\cup"] 
& \rH^2(P, p^\ast \bigwedge^2 T_F) \ar[r, "p^\ast \sigma"]
& \rH^2(P, p^\ast \omega_F^{\otimes 2} \otimes p^\ast \bigwedge^2 T_F) \\
\rH^1(P, T_P)^{\otimes 2} \ar[r, "\cup"] \ar[u, "\rd p^{\otimes 2}"]
& \rH^2(P, \bigwedge^2 T_P) \ar[u, "\bigwedge^2 \rd p"] \ar[r, "p^\ast \sigma"]
& \rH^2(P, p^\ast \omega_F^{\otimes 2} \otimes \bigwedge^2 T_P) \ar[u, "\id \otimes \bigwedge^2 \rd p", swap].
\end{tikzcd}
\]
Via the isomorphism $p^\ast \omega_F^{\otimes 2} \otimes p^\ast \bigwedge^2 T_F \simeq p^\ast \omega_F$, the composite map in the first row sends $p^\ast u \otimes p^\ast v$ to $(p^\ast \sigma)(p^\ast u \cup p^\ast v)$. Therefore, since the rightmost square in \eqref{double-square} commutes, it suffices to show that the composition of the second row with $\rH^1(F, T_F) \to \rH^1(P, T_P)$ vanishes, i.e.
\[
(p^\ast \sigma)(u_P \cup v_P) = 0 \quad \text{for all } u, v \in \rH^1(F, T_F).
\]
By definition, $p^\ast \sigma = \det \rd q$. Moreover, for any morphism $f \colon \cV \to \cW$ of vector bundles of rank $3$, the following diagram commutes:
\[
\begin{tikzcd}
\bigwedge^2 \cV  \ar[d, "\bigwedge^2 f"] \ar[rr, "\det f \otimes \id"]&  & \cHom(\det \cV, \det \cW) \otimes \bigwedge^2 \cV \\  
\bigwedge^2 \cW \ar[r, equal] & \det \cW \otimes \cW^\vee \ar[r, "\id \otimes f^\vee"]  & \det \cW \otimes \cV^\vee \ar[u, equal] 
\end{tikzcd}
\]
It follows that
\[
(p^\ast \sigma)(u_P \cup v_P)
= \det \rd q (u_P \cup v_P)
= (\rd q)^\vee \bigl(\textstyle \bigwedge^2 \rd q (u_P \cup v_P)\bigr).
\]
Since the left square in \eqref{double-square} commutes, we have
\[
\textstyle \bigwedge^2 \rd q (u_P \cup v_P)
= \rd q(u_P) \cup \rd q(v_P)
= q^\ast u_Y \cup q^\ast v_Y 
= q^\ast(u_Y \cup v_Y).
\]
The key point is that
\[
\rH^2(Y, \textstyle \bigwedge^2 T_Y) = 0,
\]
hence $u_Y \cup v_Y = 0$ for all $u, v \in \rH^1(F, T_F)$, which concludes the proof. Since $\bigwedge^2 T_Y \simeq \Omega_Y^1(2)$, this cohomological vanishing follows from the exact sequence of differentials for $Y \subset \bbP^4$, together with the standard exact sequences for the restrictions of $\Omega_{\bbP^4}^1(2)$ and $\cO_{\bbP^4}(-1)$ to $Y$, and Bott vanishing on $\bbP^4$.
\end{proof}

\begin{figure}
	\includegraphics[scale=0.4]{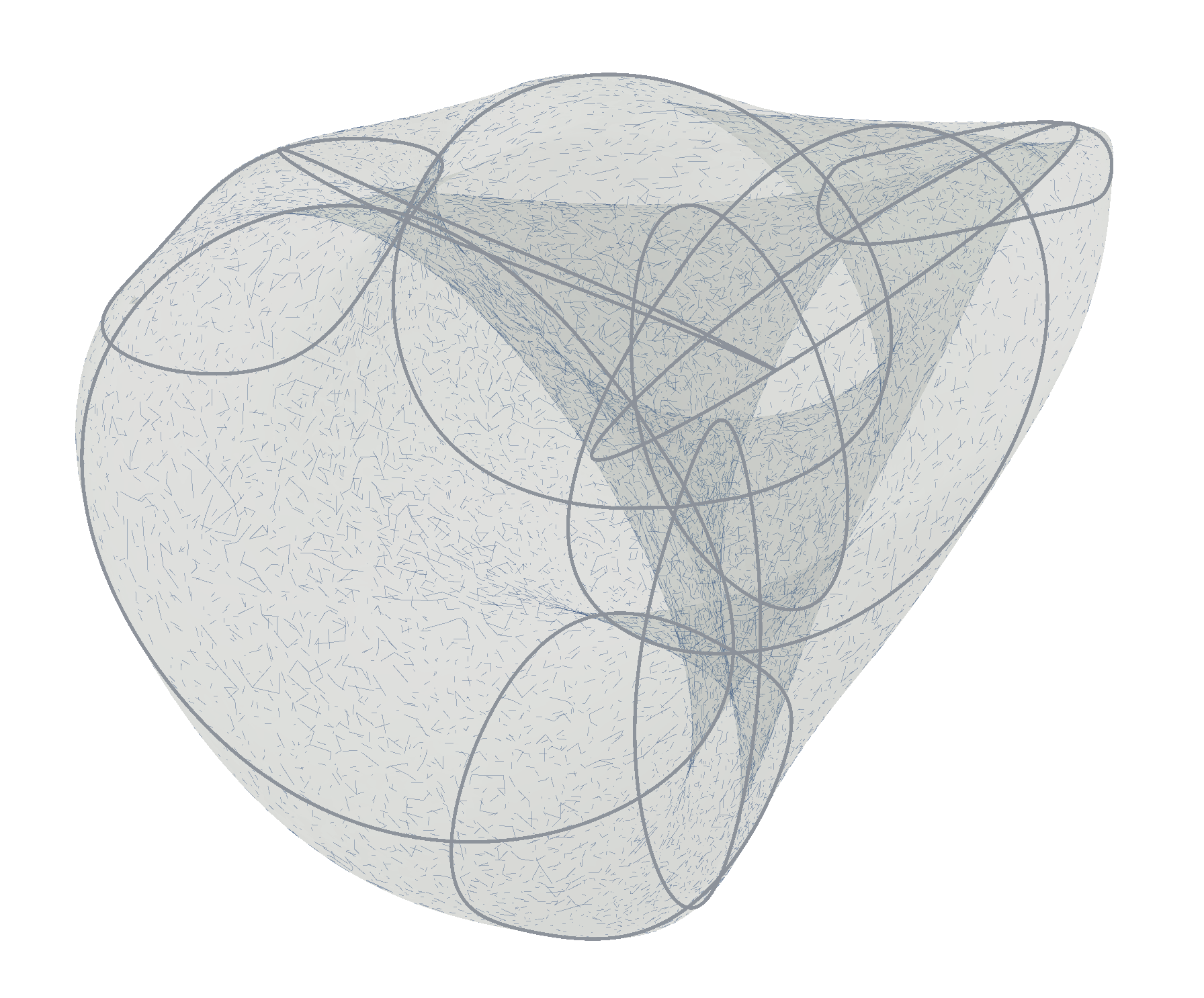}
	\caption{The real points of the Fano surface of lines on the Fermat cubic threefold, projected from its Pl\"ucker embedding to $\mathbb{R}^3$. The gray curve is the  divisor of lines of the second type, which in this case has 30 components, each smooth of genus one; 10 of these have real points. See \href{https://chocolitt.github.io/fermat_fano_real_mesh_web.html}{here} for an interactive visualization.}
\end{figure}

\subsection{Conclusion of the proof of \cref{thm:rank-of-cup-product}}
We now give a lower bound on the rank of $\mu$. In fact, we are going to compute the rank of the composite map
\[ \nu \colon \Sym^2 \rH^1(F, T_F) \stackrel{\mu}{\too} \rH^2(F, \textstyle \bigwedge^2 T_F) \stackrel{c}{\too} \Hom(\rH^0(F,\omega_F),\rH^2(F,\cO_F)), \]
where $c$ is induced by the cup product. Via Serre duality, the map $c$ can be viewed as the dual of the multiplication map
\[ \rH^0(F, \omega_F) \otimes \rH^0(F, \omega_F) \too \rH^0(F, \omega_F^{\otimes 2}). \]
This map has image $\rH^0(G, \cO_G(2)) \subset \rH^0(F, \omega_F^{\otimes 2})$ where $G = \Gr_2(E)$ is the Grassmannian of lines in $\bbP(E)$ and $\cO_G(1)$ denotes the Pl\"ucker polarization. Hence for the proof of~\cref{thm:rank-of-cup-product} it only remains to show:

\begin{proposition}\label{thm:rank-50}
For a generic cubic threefold $Y$ we have $\rk \nu =50$. In particular $\rk \mu \ge 50$.
\end{proposition}

\begin{proof} We first express everything in terms of the threefold: The incidence
correspondence induces isomorphisms $\rH^0(F, \Omega^1_F)\simeq \rH^1(Y, \Omega^2_Y)$ and
$\rH^1(F, \cO_F)\simeq \rH^2(Y, \Omega^1_Y)$. Moreover, for the Fano surface the wedge product maps
\[
\textstyle 
\bigwedge^2 \rH^0(F, \Omega^1_F) \stackrel{\sim}{\too} \rH^0(F, \omega_F),\qquad
\bigwedge^2 \rH^1(F, \cO_F)\stackrel{\sim}{\too} \rH^2(F, \cO_F),
\]
are isomorphisms~\cite[chapter 5, lemma 2.5]{huybrechts2023geometry}.
Finally, by~\cite[chapter 5, prop.~2.14]{huybrechts2023geometry} we have a natural isomorphism 
\[
\rH^1(Y,T_Y)\stackrel{\sim}{\too} \rH^1(F,T_F).
\]
Via these isomorphisms $\nu$ is identified with the cup product map
\[
\nu \colon \Sym^2 \rH^1(Y,T_Y)\to \Hom\bigl(\textstyle \bigwedge^2 \rH^1(Y, \Omega^2_Y),\bigwedge^2 \rH^2(Y, \Omega^1_Y) \bigr).
\]
We now express this in terms of the Jacobian ring of the cubic threefold $Y$. To do this, pick coordinates on $E$ and let~$F\in \bc[x_0,\dots,x_4]_3$ be a homogenous polynomial of degree three defining the smooth cubic $Y\subset \bbP^4$. Consider the Jacobian ring
\[
R(F):=\bc[x_0,\dots,x_4]/(\partial F / \partial x_i : i = 0, \dots, 4)
\]
with graded pieces $R(F)_k$ for $k \ge 0$. By Griffiths' description of the primitive cohomology of smooth hypersurfaces (see~\cite[prop.~6.2]{voisin07} for instance), there are
canonical identifications
\[
\rH^1(Y, \Omega^2_Y) \simeq R(F)_1,\qquad
\rH^2(Y, \Omega^1_Y) \simeq R(F)_4,\qquad
\rH^1(Y,T_Y) \simeq R(F)_3.
\]
Moreover, via these identifications, the cup product becomes the product of the ring $R$. With this in mind, the map $\nu$ is
\[
\nu \colon
\Sym^2 R(F)_3 \too \Hom(\textstyle \bigwedge^2 R(F)_1,\bigwedge^2 R(F)_4), \qquad f \cdot g \longmapsto [\phi \wedge \psi \mapsto f \phi \wedge g \psi +  g \phi \wedge  f \psi]. 
\]
As in \cref{sec:Fermat-cubic}, for $a, b \in R(F)$ we wrote $ab \in R(F)$ for the multiplication of the ring $R(F)$ and $a \cdot b \in \Sym^2 R(F)$ for the multiplication in the symmetric algebra over $R(F)$. Now, \cref{Prop:Fermat-rank-50} states that $\nu$ has rank $50$ for Fermat cubic threefold $Y$ with equation $F = x_0^3 + \cdots + x_4^3$. The statement then follows by semicontinuity.
\end{proof}

\section{Proof of the main theorem}

\subsection{Statements}
In this section we prove the main theorem of this paper. Let $S$ be a smooth complex variety and $\cY \to S$ a family of smooth cubic threefolds, meaning that 
\[ \cY \subset \bbP(\cE)\]
is a subvariety defined by a nonzero global section of $\Sym^3 \cE^\vee$ such that the projection $\cY \to S$ is smooth, where $\cE$ is a vector bundle of rank $5$ on $S$. Let $\pi \colon \cF \to S$ be the associated family of Fano surfaces of lines. Fix $s \in S$ and let
\[ Y:= \cY_s \quad \text{and} \quad F:= \cF_s\]
be the fibers at $s$. Statement (2) in the main theorem will be deduced from the following:

\begin{theorem} \label{Thm:conservativity} For $i = 1, 2$ let $\bbL_i$ be a rank one unitary local system over $\cF$ and $\bbV_i := \rR^2 \pi_\ast \bbL_i$. Assume the following:
\begin{enumerate}
\item the Kodaira--Spencer map $T_s S \to \rH^1(Y, T_Y)$ is surjective,\smallskip
\item the cup product map $\Sym^2 \rH^1(F, T_F) \to \rH^2(F, \bigwedge^2 T_F)$ has rank $50$,\smallskip
\item the line bundles $\bbL_i \otimes_\bbC \omega_F$ and $\bbL_i^\vee \otimes_\bbC \omega_F$ on $F$ are globally generated.\smallskip
\end{enumerate}
If the local systems $\bbV_1$ and $\bbV_2$ are isomorphic, then so are $\bbL_1 \rvert_F$ and $\bbL_2 \rvert_F$.
\end{theorem}

We now pass to statement (1) in the main theorem. Let $\bbL$ be a complex rank one local system of order $n$ on $\cF$ and consider the local system
\[ \bbV := \rR^2 \pi_\ast \bbL\]
on $S$. We are interested in computing its algebraic monodromy group, that is, the Zariski closure
\[ M := \overline{\im \rho} \; \subset \; \GL(\bbV_s) \quad \textup{of the monodromy representation} \quad \rho \colon \pi_1(S, s) \too \GL(\bbV_s).\]
Since the local system $\bbL$ is torsion, the connected component $M^\circ$ is a semisimple Lie group.
We will deduce the main theorem of the paper from the following more precise statement:

\begin{theorem} \label{Thm:precise-form-main-theorem} With the above notation, assume the following:
\begin{enumerate}
\item the Kodaira--Spencer map $T_s S \to \rH^1(Y, T_Y)$ is surjective,\smallskip
\item the cup product map $\Sym^2 \rH^1(F, T_F) \to \rH^2(F, \bigwedge^2 T_F)$ has rank $50$,\smallskip
\item for all $i \in (\bbZ / n \bbZ)^\times$ the line bundle $\bbL^{\otimes i} \otimes_\bbC \omega_F$ on $F$ is globally generated,\smallskip
\item the restriction of $\bbL$ to $F$ has order $n > 2$,\smallskip
\item the group $M$ is connected.
\end{enumerate}
Then,
\[(M, \bbV_s) \iso (E_6, V)\]
where~$V$ is one of the two irreducible representations of~$E_6$ with~$\dim V = 27$.
\end{theorem}

In the rest of this section we will prove \cref{Thm:conservativity}, then \cref{Thm:precise-form-main-theorem}, and then finally deduce the main theorem of this paper.

\subsection{Proof of \cref{Thm:conservativity}} We apply the reconstruction results in \cref{sec:ivhs} to $\pi \colon \cF \to S$. Let us recall the setup: The cup product induces a linear map
\[ \mu \colon \Sym^2 \rH^1(F, T_F) \too \rH^2(F, \textstyle \bigwedge^2 T_F)\]
which by assumption has rank $50$. The target of this map is dual to $\rH^0(F, \omega_F^{\otimes 2})$. As $h^0(F, \omega_F^{\otimes 2})=51$, it follows that the dual map
\[ \mu^\vee \colon \rH^0(F, \omega_F^{\otimes 2}) \too \Sym^2 \rH^1(F, T_F)^\vee\]
has a one-dimensional kernel. The base locus
\[ B \subset F\]
of the linear system $\ker(\mu^\vee)$ is then a bicanonical divisor (in fact, by \cref{prop:kernel-cup-product}, it is the divisor of lines of the second type)
and hence ample since $F$ is canonically polarized; see \cref{ampleness-base-locus}. It follows from \cref{reconstruction-thm} and~\cref{lem:independence-of-choices} that any isomorphism of local systems $\bbV_1 \iso \bbV_2$ induces an isomorphism of coherent sheaves 
\[ \cL_1|_B\simeq \mathscr{E}(\mathbb{V}_1) \iso \mathscr{E}(\mathbb{V}_2)\simeq \cL_2|_B \]
where $\cL_i$ is the holomorphic line bundle associated with $\bbL_i|_F$.  Since $B\subset F$ is an ample divisor, we thus have by~\cref{injectivity-pic0-restriction-ample-divisor} an isomorphism $\cL_1 \simeq \cL_2$. Since the $\bbL_i$ are unitary, it follows that $\bbL_1\rvert_F \simeq \bbL_2\rvert_F$. \qed

\subsection{Proof of \cref{Thm:precise-form-main-theorem}} We begin by recalling the notation. Let $\bbL$ be a complex rank one local system of finite order on $\cF$ and
\[ \bbV := \rR^2 \pi_\ast \bbL.\]
The algebraic monodromy group of $\bbV$ is the Zariski closure
\[ M := \overline{\im \rho} \; \subset \; \GL(\bbV_s) \quad \textup{of the monodromy representation} \quad \rho \colon \pi_1(S, s) \too \GL(\bbV_s).\]
Since the local system $\bbL$ is of finite order, the neutral component $M^\circ$ is a semisimple algebraic group. We will compare $M$ with the Tannaka group of $F$ for the convolution of perverse sheaves. To introduce it, recall that the Albanese morphism 
\[ \epsilon \colon F \too A := \Alb(F),\]
obtained by fixing some base point in $F$, is a closed embedding. We define the intersection complex 
\[ \delta_F \;:=\; \epsilon_\ast\bbC_F[2] \]
as the pushforward of the constant sheaf on $F$ under the closed immersion~$\epsilon$, shifted in cohomological degree~$-2$ so that it becomes an object of the abelian category~$\Perv(A, \bbC)$  of perverse sheaves on~$A$ as in \cite{BBDG}. The group law on the abelian variety induces a convolution product on perverse sheaves so that the convolution powers of~$\delta_F$ generate a neutral Tannaka category~$\langle \delta_F \rangle$, see for instance~\cite[sect.~3]{JKLM}. The explicit form of generic vanishing given by \cref{lem:line-bundles-coh}, guarantees that 
\begin{equation} \label{} \rH^i(F, \bbL) = 0 \qquad \textup{for $i \neq 2$},
\end{equation}
as soon as the restriction of $\bbL$ to $F$ is nontrivial. Under this assumption, the functor
\[
 \omega \colon \quad \langle \delta_F \rangle \;\too\; \Vect(\bbC), \qquad Q \longmapsto \rH^0(A, Q \otimes \bbL)
\]
is by construction a fiber functor; see~\cite[proof of th.~13.2]{KWVanishing}. The automorphisms of this fiber functor are represented by a reductive algebraic group 
\[G_{F}\; \subset \; \GL(\bbV_s)\] 
which we call the \emph{Tannaka group of~$F$}. Here we used that 
\(\omega(\delta_F) = \rH^0(A, \delta_F \otimes \bbL) = \rH^2(F, \bbL) = \bbV_s\).
We are  interested in its derived group
\[
 G_{F}^\ast \;:=\; [G_{F}^\circ, G_{F}^\circ],
\]
which is a connected semisimple algebraic group. By \cite[th. 2]{kramer2016cubic} we have
\begin{equation} \label{Eq:tannaka-is-e6}(G_F^\ast, \bbV_s) \; \iso \; (E_6, V) \end{equation}
where~$V$ is one of the two irreducible representations of~$E_6$ with~$\dim V = 27$. With this notation we have:

\begin{lemma} \label{lemma:inclusion-tannaka} Let $\bbL$ be a rank $1$ local system on $\cF$ of finite order whose restriction to $F$ is nontrivial. Then,
\[ M^\circ \subset G_F^\ast \iso E_6.\]
\end{lemma}

\begin{proof} Because of \cref{Eq:tannaka-is-e6}, the representation $\bbV_s$ of $G_F^\ast$ is irreducible. By Schur's lemma, since the group $M^\circ$ is semisimple, it suffices to show that $M^\circ$ normalizes $G_F$. This is a special case of \cite[th. 4.5]{JKLM}. Strictly speaking, the cited result is for $\ell$-adic cohomology. However, we can apply this here because the local system $\bbL$ is torsion, hence of geometric nature. 
\end{proof}

The interesting part is proving that the inclusion $M^\circ \subset G_F^\ast$ is indeed an equality. To prove this, from now on, we will assume the additional hypotheses (1) -- (5) in \cref{Thm:precise-form-main-theorem}. We will argue by contradiction and assume that $M$ is strictly contained in $G_F^\ast$. Then $M$ will be contained in a maximal semisimple subgroup
\[ H \subset G_F^\ast.\]
The maximal semisimple subgroups of $E_6$ up to conjugacy and the corresponding branching of the representation $V$ can be found for instance in~\cite[p.~298]{mckay-patera81}. We recall them in~\cref{table:branching} together with the dimensions of the occurring irreducible summands and with the data of whether or not $V$ is self-dual as a representation of $H$. Note that if in our geometric situation $\bbV_s$ were self-dual as a representation of $H$, then $\bbV_s$ would also be self-dual as a representation of the monodromy group~$M\subset H$. But then the local system $\bbV$ would be self-dual, which is ruled out by the following result:

\begin{lemma}
With assumptions as in \cref{Thm:precise-form-main-theorem}, the local system $\bbV$ is not self-dual.
\end{lemma}

\begin{proof}  If the local system $\bbV$ were self-dual, then \cref{Thm:conservativity} would yield an isomorphism of local systems $\bbL \rvert_F \iso \bbL^\vee \rvert_F$, which contradicts our assumption that $\bbL \rvert_F$ has order $n> 2$.
\end{proof}

\begin{table}[ht]
\centering
\renewcommand{\arraystretch}{1.15}
\begin{tabularx}{\linewidth}{@{}lXcc@{}}
\toprule
$H$ & Branching of $V$ & Dimensions & Self-dual? \\
\midrule
$D_5$
&
$[\,0\,] \;\oplus\; [\varpi_1] \;\oplus\; [\varpi_4]$
& $1 \oplus 10\oplus 16$
& No 
\\\midrule

$C_4$
&
$[\varpi_2]$
& $27$
& Yes
\\\midrule

$F_4$
&
$[\, 0\,] \;\oplus\; [\varpi_4]$
& $1\oplus 26$
& Yes
\\\midrule

$A_2$
&
$[2\varpi_1 + 2\varpi_2]$
& $27$
& Yes
\\\midrule

$G_2$
&
$[2\varpi_2]$
& $27$
& Yes
\\\midrule

$A_2\times G_2$
&
$[\varpi_2 \times \varpi_2] \;\oplus\; [(2\varpi_1) \times 0\,]$
& $21\oplus 6$
& No
\\\midrule

$A_1\times A_5$
&
$\displaystyle
[{\varpi_1} \times {\varpi_5}] \;\oplus\; [\,0 \times \varpi_2]$
& $12\oplus 15$
& No
\\\midrule

$A_2\times A_2\times A_2$
&
$[\varpi_1\times \varpi_1 \times 0\,]
\;\oplus\; 
[\varpi_2 \times 0 \times \varpi_1]
\;\oplus\;
[\,0\times \varpi_2\times \varpi_2]$
& $9\oplus 9\oplus 9$
& No
\\
\bottomrule \smallskip
\end{tabularx}
\caption{The  maximal connected semisimple subgroups $H\subset E_6$ and the branching rules for the restriction of $V$ to $H$. In the second column we denote by $[\lambda]$ the irreducible representation of highest weight $\lambda$. For each simple Dynkin type we denote by $\varpi_1, \varpi_2, \dots$ the fundamental dominant weights which form the basis dual to the coroots of the simple positive roots, numbered as in~\cite[p.~2]{mckay-patera81}.}
\label{table:branching} 
\end{table}

Ruling out the maximal semisimple subgroups $H$ for which the representation $V$ is not self-dual requires a finer argument. The key point is that the local system $\bbV$ contains a unique nontrivial simple local subsystem and that we can bound its rank:% following result:

\begin{lemma} 
With assumptions as in \cref{Thm:precise-form-main-theorem}, we have
\[
 \bbV \;\iso\; \bbW \oplus \bbW'
\]
for some simple local system $\bbW$ of rank $\rk(\bbW) \ge 13$ and trivial $\bbW' \iso \bbC_S^m$ for some $m\ge 0$.
\end{lemma}

\begin{proof} This follows from \cref{supplement-finite-monodromy}, from which we borrow notation. Under the assumptions of the theorem, as in the proof of \cref{Thm:conservativity}, the base locus $B \subset F$ is a bicanonical divisor, hence both $B$ and $\omega_F^\vee(B)\simeq \omega_F$ are ample; see \cref{ampleness-base-locus}. As $B$ is ample, we have $H^0(B, \mathscr{O}_B)=\mathbb{C}$. Moreover, for any $i \in (\bbZ / n \bbZ)^\times$ the line bundle $\bbL^{\otimes i} \otimes_\bbC \omega_F$ is globally generated. It follows that we can consider the sub local system
\[ \bbW_i \subset \bbV_i := \rR^2 \pi_\ast \bbL^{\otimes i}\]
given by \cref{th:consequences-monodromy}. By that theorem, this local system has rank $$\rk \bbW_i \ge  h^0(F, \mathbb{L}^{\otimes i}\otimes \omega_F)+h^0(F, (\mathbb{L}^\vee)^{ \otimes i} \otimes \omega_F)+1 = 13.$$  

Now, by \cref{lemma:inclusion-tannaka}, the connected component of algebraic monodromy group of $\bbV_i$ is contained in $E_6$. By ~\cref{table:branching}, the restriction of $V$ to any semisimple subgroup $H \subsetneq E_6$ with $V|_H$ non-self-dual has at most one irreducible direct summand of dimension $> 12$. \Cref{supplement-finite-monodromy} then implies that $\mathbb{V}\simeq \mathbb{W}\oplus \mathbb{W}'$ where $\mathbb{W}$ is simple of rank at least $13$ and $\mathbb{W}'$ has finite monodromy; as we are assuming $M$ is connected, this implies $\mathbb{W}'$ is trivial as desired.
\end{proof}

To conclude the proof of \cref{Thm:precise-form-main-theorem}, we rule out the non-self-dual cases one by one. From now on we identify $(G_F^\ast, \bbV_s)$ with $(E_6, V)$ and consider the fibers 
 $W = \bbW_s$, $W'= \bbW'_s$ of $\bbW$, $\bbW'$ as subrepresentation of $V$ via $\bbV_s \iso V$. 

To begin with, note that the Dynkin type of $H$ cannot be $A_2 \times A_2 \times A_2$ because in that case $V$ splits as three irreducible representations of $H$ of dimension $9$, but $W$ has dimension at least $13$.

If $H$ had Dynkin type $D_5$, the representation $W$ would be contained in the $16$-dimensional representation of $H$ appearing in the table, hence $W'$ would contain the $10$-dimensional representation. Now this representation of $H$ is almost faithful (it is the standard representation of $\SO_{10}$), hence $W'$ is an almost faithful representation of $M$. But $M$ acts trivially on $W'$, thus $M$ would be trivial. This is a contradiction because $W$ is an irreducible representation of $M$ of dimension $>1$. 

Reasoning similarly, if $H$ had Dynkin type $A_1 \times A_5$, the representation $W'$ would contain the~$12$-dimensional representation of $H$ appearing in the table. Again, this is almost faithful (it is the tensor product of the standard representations of $\SL_2$ and $\SL_6$). Thus we would conclude that $W'$ is an almost faithful representation of $M$, hence $M$ is trivial. Contradiction.

Finally, if $H$ had Dynkin type $A_2 \times G_2$, then $W'$ would contain the $6$-dimensional representation of~$H$ appearing in the table. This representation is the symmetric square of the standard representation of $\SL_3$ on which $G_2$ acts trivially. Since $W'$ is the trivial representation of $M$, we would have that $M$ is contained in the factor $G_2$. On the other hand, the $21$-dimensional representation of~$H$ appearing in the table is the tensor product of the standard representation of $\SL_3$ and the~$7$-dimensional representation of $G_2$. Hence the irreducible summands of its restriction to $G_2$ have dimension $7$. But $W$ is contained in this representation, it is irreducible and has dimension $\ge 13$, which is impossible if $M \subset G_2$.

Summing up, the group $M$ cannot be contained in any of the maximal semisimple groups of $E_6$, hence $M = G_F^\ast \iso E_6$. This concludes the proof of \cref{Thm:precise-form-main-theorem}. \qed

\subsection{Proof of the main theorem} We are now in position to prove the main theorem in \cref{sec:main-results}. First note that by \cref{thm:global-generation-line-bundle}, there is some $n_0 > 2$ with the following property: given a very general smooth cubic threefold $Y \subset \bbP^4$, we have that any line bundle $\cL$ on $F$ of finite order $n \ge n_0$ is such that $\cL \otimes \omega_F$ is globally generated. Indeed, let $\eta$ be the generic point of the moduli space of smooth cubic threefolds, $Y_\eta$ the generic cubic $3$-fold, and $F_\eta$ its Fano surface of lines. By \cite[chapter 1, 2.13(i)]{huybrechts2023geometry}, the intermediate Jacobian of a very general cubic $3$-fold, and hence the (isomorphic, by \cite[chapter 5, corollary 3.3]{huybrechts2023geometry}) Albanese of the corresponding Fano surface of lines, is simple; thus the Albanese of $F_\eta$ is geometrically simple. Hence \cref{thm:global-generation-line-bundle} implies that there exists $n_0$ such that any line bundle $\cL$ on $F$ of finite order $n \ge n_0$  satisfies that $\cL \otimes \omega_{F_\eta}$ is globally generated. For each $n\geq n_0$ the set of cubic threefolds whose Fano surface of lines $F$ admits an $n$-torsion line bundle $\cL$ such that $\cL\otimes \omega_F$ fails to be globally generated is hence contained in a proper closed subset of moduli. Thus the same $n_0$ suffices for a very general $Y$ as claimed.

To prove statement (1) of the main theorem the key point is that, the base $S$ being smooth, for any \'etale morphism $S' \to S$ (not necessarily finite) the induced map between topological fundamental groups has image a subgroup of finite index in $\pi_1(S)$. In particular, the connected component of the algebraic monodromy group is insensitive to passing such an $S'$. 

Choosing very general $s\in S$ and setting $Y=\cY_s, F=\cF_s$, we may assume that
\begin{enumerate}
\item the Kodaira--Spencer map $T_s S \to \rH^1(Y, T_{Y})$ is surjective, \smallskip
\item the cup product map $\Sym^2 \rH^1(F, T_F) \to \rH^2(F, \bigwedge^2 T_F)$ has rank $50$, \smallskip
\stepcounter{enumi}
\item the restriction of $\cL$ to $F$ has order $n \ge n_0$,
\end{enumerate}
where $\cL$ is the flat line bundle associated with the rank one local system $\bbL$. Indeed, under the assumption that the family $\cY \to S$ is versal, the first condition holds for generic $s$ by definition of a versal family; the second holds for generic $s$ by \cref{thm:rank-of-cup-product}; and the last holds for generic~$s$ by assumption in the statement of the main theorem.
Moreover, as $s$ was very general we may assume by the first paragraph of this proof that all line bundles $\cL$ on $F$ of order $n$ satisfy that $\cL\otimes \omega_F$ is globally generated, and in particular that
\begin{enumerate} \setcounter{enumi}{2}
\item for all $i \in (\bbZ / n \bbZ)^\times$ the line bundles $\cL^{\otimes i} \otimes \omega_F$ on $F$ are globally generated.
\end{enumerate}

 Passing to a finite \'etale cover of $S$, we may further assume that the rank one local system $\bbL$ on~$\cF$ has order $n$ and that
 \begin{enumerate} \setcounter{enumi}{4}
 \item the algebraic monodromy group of $\bbV := \rR^2 \pi_\ast \bbL$ is connected.
 \end{enumerate} 
The assumptions of \cref{Thm:precise-form-main-theorem} are thus satisfied, and the monodromy statement in (1) follows.

We now prove statement (2) of the main theorem. For $i = 1, 2$ we may assume as above that $\bbL_i$ is a local system of finite order $n_i \ge n_0$ on $\cF$ whose restriction to the fibers of $\pi \colon \cF \to S$ has order $n_i$. By the choice of $n_0$, the associated line bundle $\cL_i$ is then such that both $\cL_i \otimes \omega_F$ and $\cL_i^\vee \otimes \omega_F$ are globally generated. If the local systems $\bbV_1 = \rR^2 \pi_\ast \bbL_1$ and $\bbV_2 = \rR^2 \pi_\ast \bbL_2$ are isomorphic, then \cref{Thm:conservativity} implies that the restrictions $\bbL_1 \rvert_{\cF_s}$ and $\bbL_2 \rvert_{\cF_s}$ are isomorphic.

It remains to prove the statement about the invariant trace field $\inv(\rho_s)$ of $\rho_s$. First, note that
\[
\bbQ(\rho) \subset K := \bbQ(\zeta_n),
\]
where $\zeta_n = e^{2\pi i/n}$. Indeed, the local system $\bbL$ is obtained by extension of scalars from a local system $\bbL_K$ defined over $K$. It follows that $\bbV = \bbV_K \otimes_K \bbC $ with  $\bbV_K = \rR^2 \pi_\ast \bbL_K$ so the traces of $\rho_s$ lie in $K$, giving the inclusion. For the converse inclusion, after replacing $S$ by a finite \'etale cover, we may assume that
\[
\inv(\rho_s) = \bbQ(\rho_s).
\]
For $\sigma \in \Gal(K / \bbQ)$, consider the conjugated local systems $\bbL_K^\sigma := \bbL_K \otimes_{\sigma} K$ and the monodromy representation $\rho_s^\sigma$ associated with $\bbV_K^\sigma := \bbV_K \otimes_\sigma K = \rR^2 \pi_\ast \bbL_K^\sigma$.
Suppose by contradiction that there exists a nontrivial $\sigma \in \Gal(K/\bbQ)$ acting trivially on $\bbQ(\rho_s)$. Then $\rho_s$ and its Galois conjugate $\rho_s^\sigma$ have the same traces, hence are isomorphic (since $\rho_s$ is semisimple). It follows that the corresponding local systems $\bbV_K$ and $\bbV_K^\sigma$ are isomorphic. By (the already-proven) statement (2) of the main theorem, this implies that, for a fiber $\cF_s$, the restrictions of the underlying rank-one local systems satisfy
\[
\bbL|_{\cF_s} \cong \left(\bbL|_{\cF_s}\right)^{\otimes r},
\]
where $r \in (\bbZ/n\bbZ)^\times$ is determined by $\sigma(\zeta_n) = \zeta_n^r$, since $\bbL_K^\sigma \cong \bbL_K^{\otimes r}$. This contradicts the assumption that $\bbL$ has order $n$, and thus concludes the proof of the main theorem. \qed

\bibliographystyle{alpha}
\bibliography{bibliography-e6.bib}

\end{document}